\documentclass[reqno]{amsart}
\usepackage{amsmath,amssymb,bbm,url}
\newtheorem{theorem}{Theorem}[section]
\usepackage[scr=boondoxo]{mathalfa}
\newtheorem{lemma}[theorem]{Lemma}

\theoremstyle{remark}
\newtheorem{remark}[theorem]{Remark}

\newcommand{\diag}{\operatorname{diag}}
\numberwithin{equation}{section}

\newcommand{\kG}{{\mathcal G}}
\newcommand{\kX}{{\mathcal X}}
\newcommand{\kE}{{\mathcal E}}
\newcommand{\kY}{{\mathcal Y}}
\newcommand{\kB}{{\mathcal B}}
\newcommand{\kC}{{\mathcal C}}
\newcommand{\kV}{{\mathcal V}}
\newcommand{\kW}{{\mathcal W}}
\newcommand{\kD}{{\mathcal D}}
\newcommand{\kL}{{\mathcal L}}
\newcommand{\kH}{{\mathcal H}}
\newcommand{\kU}{{\mathcal U}}

\newcommand{\kI}{{\mathcal I}}

\newcommand{\kZ}{{\mathcal Z}}
\newcommand{\kS}{{\mathcal S}}

\newcommand{\kh}{{\mathscr h}}
\newcommand{\kl}{{\mathscr l}}

\newcommand{\kCb}{\kC_{\operatorname{b}}}

\newcommand  {\D}{{\mathrm D}}
\renewcommand{\d}{{\mathrm d}}
\renewcommand{\i}{{\mathrm i}}
\newcommand  {\e}{{\mathrm e}}
\newcommand  {\p}{{\mathrm p}}

\renewcommand{\a}{{\mathsf a}}
\newcommand {\mS}{{\mathsf S}}
\newcommand {\ms}{{\mathsf s}}
\newcommand {\mg}{{\mathsf g}}
\newcommand {\mG}{{\mathsf G}}
\newcommand {\mP}{{\mathsf P}}
\newcommand {\mQ}{{\mathsf Q}}
\renewcommand{\b}{{\mathsf b}}

\newcommand  {\1}{\mathbbm{1}}
\newcommand  {\R}{{\mathbb R}}
\newcommand  {\N}{{\mathbb N}}
\newcommand  {\C}{{\mathbb C}}
\newcommand  {\JJ}{{\mathbb J}}   
\newcommand  {\Z}{{\mathbb Z}}
\renewcommand{\P}{{\mathbb P}}
\newcommand  {\Q}{{\mathbb Q}}
\renewcommand{\S}{{\mathbb S}}

\newcommand {\id}{\operatorname{id}}

\newcommand {\spec}{\operatorname{spec}}

\renewcommand {\Re}{\operatorname{Re}}
\renewcommand {\Im}{\operatorname{Im}}
\newcommand{\dist}{\operatorname{dist}}
\newcommand{\interior}{\operatorname{int}}

\newcommand{\norm}[3][{\vphantom 1}]{\lVert #2 \rVert_{#3}^{#1}}
\newcommand{\snorm}[3][{\vphantom 1}]{\lvert #2 \rvert_{#3}^{#1}}
\def\vvvert{\hbox{\ensuremath{|\hspace{-0.16em}|\hspace{-0.16em}|}}}
\newcommand{\Norm}[3][{\vphantom 1}]%
  {\vvvert #2 \vvvert_{#3}^{#1}}

\makeatletter
\def\@sect#1#2#3#4#5#6[#7]#8{%
  \edef\@toclevel{\ifnum#2=\@m 0\else\number#2\fi}%
  \ifnum #2>\c@secnumdepth \let\@secnumber\@empty
  \else \@xp\let\@xp\@secnumber\csname the#1\endcsname\fi
  \@tempskipa #5\relax
  \ifnum #2>\c@secnumdepth
    \let\@svsec\@empty
  \else
    \refstepcounter{#1}%
    \edef\@secnumpunct{%
      \ifdim\@tempskipa>\z@ 
        \@ifnotempty{#8}{.\@nx\enspace}%
      \else
        \@ifempty{#8}{.}{.\@nx\enspace}%
      \fi
    }%
    \@ifempty{#8}{%
      \ifnum #2=\tw@ \def\@secnumfont{\bfseries}\fi}{}%
    \protected@edef\@svsec{%
      \ifnum#2<\@m
        \@ifundefined{#1name}{}{%
          \ignorespaces\csname #1name\endcsname\space
        }%
      \fi
      \@seccntformat{#1}%
    }%
  \fi
  \ifdim \@tempskipa>\z@ 
    \begingroup #6\relax
    \@hangfrom{\hskip #3\relax\@svsec}{\interlinepenalty\@M #8\par}%
    \endgroup
    \ifnum#2<\@m \@tocwrite{#1}{#7}\fi
  \else
  \def\@svsechd{#6\hskip #3\@svsec
    \@ifnotempty{#8}{\ignorespaces#8\unskip
       \@addpunct.}%
    \ifnum#2<\@m \@tocwrite{#1}{#7}\fi
  }%
  \fi
  \global\@nobreaktrue
  \@xsect{#5}}
\makeatother

\begin{document}
\title[Exponentially accurate embeddings of time semidiscretizations]
{Exponentially accurate Hamiltonian embeddings of symplectic A-stable
Runge--Kutta methods for Hamiltonian semilinear evolution equations}

\author[C. Wulff]{Claudia Wulff}
\address[C. Wulff]%
{Department of Mathematics\\
 University of Surrey \\
 Guildford GU2 7XH \\
 UK}
\email{c.wulff@surrey.ac.uk}

\author[M. Oliver]{Marcel Oliver}
\address[M. Oliver]%
{School of Engineering and Science \\
 Jacobs University \\
 28759 Bremen \\
 Germany}
\email{oliver@member.ams.org}

\date{\today}

\begin{abstract}
We prove that a class of A-stable symplectic Runge--Kutta time
semidiscretizations (including the Gauss--Legendre methods) applied to
a class of semilinear Hamiltonian PDEs which are well-posed on spaces
of analytic functions with analytic initial data can be embedded into
a modified Hamiltonian flow up to an exponentially small error.  As a
consequence, such time-semidiscretizations conserve the modified
Hamiltonian up to an exponentially small error.  The modified
Hamiltonian is $O(h^p)$-close to the original energy where $p$ is the
order of the method and $h$ the time step-size.  Examples of such
systems are the semilinear wave equation or the nonlinear
Schr\"odinger equation with analytic nonlinearity and periodic
boundary conditions.  Standard Hamiltonian interpolation results do
not apply here because of the occurrence of unbounded operators in the
construction of the modified vector field. This loss of regularity in
the construction can be taken care of by projecting the PDE to a
subspace where the operators occurring in the evolution equation are
bounded and by coupling the number of excited modes as well as the
number of terms in the expansion of the modified vector field with the
step size.  This way we obtain exponential estimates of the form
$O(\exp(-c/h^{1/(1+q)}))$ with $c>0$ and $q \geq 0$; for the
semilinear wave equation, $q=1$, and for the nonlinear Schr\"odinger
equation, $q=2$.  We give an example which shows that analyticity of
the initial data is necessary to obtain exponential estimates.
\end{abstract}

\maketitle
\
\tableofcontents

\section{Introduction}

Neishtadt \cite{Neishtadt1984} showed that a system of ordinary
differential equations (ODEs) with a rapidly rotating phase of period
$\epsilon$ can be transformed into a system of ODEs for the slow
variables, and a scalar ODE for the fast phase, both of which are, up
to a small error, independent of the fast phase variable.

When the vector field is analytic in the slow coordinates the
procedure can be carried out up to an exponentially small remainder of
magnitude $O(\e^{-c/\epsilon})$ for some $c>0$.  The result is proved
by applying $N$ near-identity transformations, each of which improves
the error by one order in $\epsilon$, carefully estimating the
remainders, and concluding that the embedding is optimal when $N =
O(1/\epsilon)$.  For rapidly forced Hamiltonian ODEs this implies
approximate conservation of the Hamiltonian of the truncated slow
system in the new coordinates with an error of order
$O(\e^{-c/\epsilon})$ over times of length $O(1)$ and, in particular,
approximate conservation of the averaged Hamiltonian of the original
slow system with error $O(\epsilon)$ over exponentially long times
provided that the trajectory of the system remains bounded.

An $O(\epsilon)$-close to identity analytic map $\psi^\epsilon$ on a
finite dimensional space $\R^n$ is the time-$\epsilon$ map of a
rapidly forced analytic vector field where $\epsilon$ is the period of
the forcing.  Neishtadt's result therefore applies and shows that
$\psi^\epsilon$ can be embedded into the flow of a system of
autonomous ODEs up to an exponentially small error.  When
$\psi^\epsilon$ is symplectic, this flow is also symplectic.  This
proves that the iterates of $\psi^\epsilon$ approximately conserve the
energy of this flow over exponentially long times so long as they
remain bounded.
 
Benettin and Giorgilli \cite{BeGi1994} give an alternative proof of
this embedding result by matching a Taylor expansion of a
diffeomorphism $\psi^\epsilon$ with the formal power series expansion
of the flow of an $\epsilon$-dependent vector field $\tilde f(x)$,
called modified vector field \cite{HaLu2002,LeRe04}, truncating at
some order $N$, where, as before, the embedding is optimal when $N =
O(1/\epsilon)$. 

In particular, if the diffeomorphism $\psi^\epsilon$ is a one-step
discretization of order $p$ with step size $h=\epsilon$ of an ODE
$\dot x = f(x)$, then $\psi^\epsilon$ can be approximately embedded
into the flow of a modified vector field $\tilde f$ which satisfies
$\lVert f-\tilde f \rVert=O(h^p)$.  If the ODE is Hamiltonian with
energy $H$ and $\psi^\epsilon$ is symplectic, then the modified vector
field is also Hamiltonian with energy $\tilde H$ and $\lvert H -
\tilde H \rvert = O(h^p)$.  This implies approximate conservation of
the energy $H$ of the original system by the symplectic time-stepping
method $\psi^\epsilon$ over exponentially long times provided the
numerical trajectory remains bounded.  This strategy has been used to
prove approximate energy conservation of many classes of symplectic
numerical methods, in particular symplectic Runge--Kutta
discretizations, \cite{HaLu2002,LeRe04,Re1999} and references therein.

The question arises whether and to what extent these results extend to
partial differential equations (PDEs).  Here, the phase space is
typically an infinite dimensional Hilbert space and the vector field
contains unbounded operators, usually in the form of spatial
derivatives.  These unbounded operators propagate into the transformed
vector field of Neishtadt \cite{Neishtadt1984} and into the formal
series expansion of the modified vector field of Benettin and
Giorgilli \cite{BeGi1994}.
 
Note that the analytic difficulties persist when analyzing full
space-time discretizations of the problem.  When discretizing space,
unbounded operators turn into a sequence of bounded operators whose
operator norms diverge as the spatial resolution increases.
Consequently, the constant $c$ in the exponential error estimate
$O(e^{-c/\epsilon})$ tends to $0$ with increasing spatial resolution,
so that the approximate embedding result for general initial data
without the requirement of high regularity fails unless $\epsilon$ is
coupled to the spatial step size in a suitable way.  This leads to
severe restrictions on the time step size.  In the case of hyperbolic
problems such as semilinear wave equations, the naive approach fails
for step size ratios close to the CFL limit, i.e., in the practically
relevant regime.

Matthies and Scheel \cite{MaSch2001} consider semilinear Hamiltonian
PDEs coupled to a high frequency oscillator via a nonlinearity which
is bounded on the underlying Hilbert space.  They prove that
Neishtadt's result of an approximate embedding into a flow where the
slow variables are decoupled from the fast oscillator is still true,
albeit with an error $O(\exp(- c/h^{1/(q+1)}))$ under the condition
that the initial data is in a Gevrey class associated with the
evolution equation.  For the semilinear wave or nonlinear
Schr\"odinger equation, this amounts to requiring real analyticity of
the initial data.  The positive constants $q$ and $c$ in the error
estimate depend on the Gevrey class associated with the evolution
equation.  Matthies proves a similar result for rapidly forced
parabolic PDEs \cite{Matthies01} and an approximate embedding result
for space-time discretizations of parabolic PDEs \cite{Matthies03}.
The paper \cite{MaSch2001} also provides an example which shows that
Gevrey regulary initial data is necessary.  The conclusion is that
exponential averaging still works in this context, but long time
approximate conservation of the averaged energy might fail because
solutions of Hamiltonian evolution equations are not generally Gevrey
regular over long times.

The aim of this paper is to prove a similar result for
time-semidiscretizations of PDEs.  Note that while formally a time
discretization of a PDE can be embedded into the flow of a rapidly
forced evolution equation via the construction of Fiedler and Scheurle
\cite[Section 2]{FiedlerScheurle}, the rapidly forced nonlinear term
in the interpolating evolution equation is not bounded, so that the
results of Matthies and Scheel \cite{MaSch2001} do not apply.

Runge--Kutta time-semidiscretizations of semilinear Hamiltonian PDEs
are only well-defined if the method is implicit; explicit or partially
implicit Runge--Kutta time-semidiscretizations such as partitioned
Runge--Kutta methods, the simplest of which are the leap frog and
symplectic Euler schemes, cannot satisfy the CFL condition for any
size of the time step \cite{AriehBook}.  Thus, in this paper we
consider a class of (implicit) symplectic A-stable Runge--Kutta
methods which includes the Gauss--Legendre Runge--Kutta methods.  The
simplest of these methods is the implicit midpoint rule.

The analysis relies on our earlier work: In \cite{OliverWulffRK} we
analyzed the differentiability properties with respect to the initial
value and time of the semiflow of
\begin{equation}
  \label{e.pde}
  \partial_t U  = F(U)  = AU + B(U) 
\end{equation}
on a scale of Hilbert spaces, and obtained analogous results for the
time-$h$ map of its corresponding A-stable Runge--Kutta
time-semidiscretization.  In \cite{OliverWulffGalerkin}, we prove
stability of the semiflow and of the time-semidiscrete solution under
spatial spectral Galerkin approximation.
 
Our approach applies to a large class of semilinear Hamiltonian PDEs
with analytic nonlinearities, including the semilinear wave equation
and the nonlinear Schr\"odinger equation on the circle.  Our main
result, Theorem~\ref{t.GenMain}, can be paraphrased as follows.  If a
semilinear Hamiltonian evolution equation with energy $H$ is
discretized by a symplectic A-stable Runge--Kutta method $\Psi^h$ of
order $p$, then there exists a modified Hamiltonian flow
$\tilde{\Phi}$, defined for Gevrey regular data, with a Hamiltonian
$\tilde H$ that is $O(h^p)$ close to $H$, such that $\tilde\Phi^h$
interpolates $\Psi^h$ with exponentially small error $O(\exp(-
c_*/h^{1/(q+1)}))$.  As a consequence, the modified energy $\tilde H$
is conserved by the symplectic integrator $\Psi^h$ with the same
exponentially small error for Gevrey-regular initial data.  This
result is in a number of ways parallel to what has been proved for
rapidly forced PDEs of Matthies and Scheel \cite{MaSch2001}: as in
their work, the constants $q$ and $c_*$ depend on the Gevrey class
associated with the evolution equation, and $c_*$ also depends on the
numerical scheme; moreover, as in \cite{MaSch2001}, the result does
not imply long time approximate energy conservation for time
semi-discretizations because both the solutions of the PDE and the
numerical trajectories are typically not Gevrey regular over long
times.

Let us mention some related work. Moore and Reich \cite{MooreReich02}
derive a modified multisymplectic partial differential equation for a
multisymplectic discretization of the semilinear wave equation which
is satisfied by the numerical solution with higher accuracy than the
discretization error; further results in this direction are due to
Islas and Schober \cite{IslasSchober}.  Both papers derive higher
order modified equations, but leave it open whether these are well
posed.  In this paper, we can actually prove that the interpolating
flow is well-defined.  In \cite{OWW04}, it is shown that second order
finite difference space-semidiscretizations of analytic solutions of
the semilinear wave equation approximately conserve a discrete
momentum map up to an exponentially small error.  Cano \cite{Cano}
considers symplectic space-time discretizations of semilinear wave
equations and constructs a finite order modified Hamiltonian, assuming
certain conjectures on the smoothness of the fully discrete system.

The approximate conservation of invariants by splitting methods for
Hamiltonian PDEs has been studied extensively via normal form
transformations or modulated Fourier expansions.  Such results require
less stringent regularity assumptions, but they are limited either to
linear equations \cite{DebusE09,Faou07} or to the weakly nonlinear
regime, i.e., to small initial data of space-time discretizations near
a homogeneous equilibrium, \cite{CohenEtAl,FaouII-10, Gauckler10}, or
they refer to modified numerical methods which dampen high
oscillations \cite{Faou09, FaouII-10}.  Note that (symplectic)
Gauss--Legendre Runge--Kutta discretizations of linear Hamiltonian
systems preserve energy exactly \cite{HaLu2002}.  The results of
\cite{CohenEtAl, FaouI-10, FaouII-10, Gauckler10} give approximate
conservation of actions and regularity of trajectories of splitting
methods applied to semilinear wave and Schr\"odinger equations for
small initial data over polynomially long times under non-resonance
conditions. These results require initial values in high order Sobolev
spaces \cite{CohenEtAl, FaouII-10, Gauckler10} or restrictive
conditions on the coupling between space and time step size
\cite{FaouI-10}.  In \cite{Faou09}, exponentially accurate
interpolations are constructed for modified splitting methods which
dampen highly oscillatory motion.

Exponentially accurate estimates for PDEs, albeit without reference to
a Hamiltonian structure, have also been obtained in the context of
homogenization of linear elliptic problems for Gevrey regular data
\cite{KMS}, while a result for homogenization up to all orders for
nonlinear elliptic PDEs can be found in \cite{CS04}.
  
The paper is structured as follows.  Section~\ref{s.pde} defines the
precise class of semilinear Hamiltonian PDEs which we study in this
paper.  This class includes the semilinear wave equation and the
nonlinear Schr\"odinger equation.  In Section~\ref{s.rk}, we introduce
A-stable symplectic Runge--Kutta methods.  These methods are
well-defined on Hilbert spaces when applied to a semilinear PDE of the
class considered.  In Section~\ref{s.bea}, we present and prove our
main result, Theorem~\ref{t.GenMain}, on approximate Hamiltonian
interpolation of the time-$h$ map of such Runge--Kutta methods.
Finally, in Section~\ref{s.LowerEst}, we given an example of a
nonlinear Schr\"odinger equation in Fourier space which shows that
Gevrey regular initial data are necessary for an exponentially
accurate embedding of a symplectic Runge--Kutta method into a flow.


\section{Semilinear Hamiltonian PDEs}
\label{s.pde}

In this section, we describe the class of semilinear Hamiltonian
systems on Hilbert spaces considered in this paper.  We begin by
reviewing the general functional setting for semilinear evolution
equations from \cite{Pazy} and introduce Gevrey spaces.  In
Section~\ref{s.differentiability}, we review results on the
differentiability in time of the semiflow from \cite{OliverWulffRK}.
In Section~\ref{s.hamiltonian}, we restrict to the Hamiltonian case
and review a well-known integrability lemma in our Hilbert space
setting.  Section~\ref{ss.analyticFcts} introduces Hilbert spaces of
analytic functions and superposition operators on these spaces.
Finally, in Sections~\ref{ss.swe} and \ref{ss.nse}, respectively, we
show how our main examples, the nonlinear Schr\"odinger equation and
the semilinear wave equation, fit into this framework.

\subsection{Semilinear evolution equations}
\label{ss.genSett}

We initially consider an abstract semilinear evolution equation of the
form \eqref{e.pde},
\[
  \partial_t U  = F(U)  = AU + B(U)
\]
on a Hilbert space $\kY$.  We assume the following.

\begin{itemize}
\item[(A)] $A$ is a normal operator on a Hilbert space $\kY$ which
generates a $\kC^0$-semigroup $\e^{tA}$.
\end{itemize}

Recall that an operator $A$ is normal if it is closed and $AA^* = A^*
A$.  For a definition of strongly continuous semigroups
($\kC^0$-semigroups), see \cite{Pazy}.  Condition (A) implies that
there is a constant $\omega>0$ such that $\lVert \e^{tA} \rVert \leq
\e^{t\omega}$ for all $t\geq 0$.
 
To formulate our assumptions on the nonlinearity $B$, we need some
definitions.  We write
\[
  \kB^\kX_R (U^0) =  \{ U \in \kX \colon \norm{U-U^0}{\kX} \leq R \} 
\]
to denote the closed ball of radius $R$ in a Hilbert space $\kX$ about
$U^0 \in \kX$.  When no confusion about the space is possible, we may
drop the superscript $\kX$.  Let $\kD \subset \kY$ be open.  For
$\delta>0$, let
\[
  \kD^{\delta} 
  = \bigcup_{U \in \kD} \kB_\delta^{\kY}(U) \,.
\]
Let $\kY^\C \equiv \kY + \i \kY$ denote the complexification of $\kY$.
We define, for fixed $\delta>0$,
\[
  \kD^\C = \bigcup_{U \in \kD} \kB_\delta^{\kY^\C}(U) \,.
\]
Our assumption on $B$ is then stated as follows.
\begin{itemize}
\item[(B0)] There is some $\delta>0$ and a bounded open set $\kD
\equiv \kD_0$ such that $B \colon \kD^\C \to \kY^\C$ is analytic with
bound $M_0$.
\end{itemize}

Then, after casting \eqref{e.pde} in its \emph{mild formulation}
\begin{equation}
  \label{e.mild}
  U(t) = \e^{tA} U^0 + \int_0^t \e^{(t-s)A} \, B(U(s)) \, \d s \,,
\end{equation}
we can apply the contraction mapping theorem with parameters to obtain
well-posedness locally in time \cite{Pazy}.  Let $U^0 \to \Phi^t(U^0)$
denote the flow of \eqref{e.mild}, i.e., $U(t) = \Phi^t(U^0)\in
\kD^\C$ satisfies \eqref{e.mild} with $U(0) = U^0 \in \kD^\C$.  Then
$\Phi^t$ is continuous in $t$ and analytic in $U^0$.

For $m \in \N$, let $\P_m$ denote the sequence of spectral projectors
of $A$ onto the set $\kB^\C_m(0) \cap \spec A$, set $\P \equiv \P_1$,
and $\Q \equiv 1-\P$.  Assumption (A) implies that
\[
  \lim_{m\to \infty} \P_m U = U
\]
for all $U \in \kY$, and that
\begin{equation}
  \label{eq:PmAEstimate}
  \norm{A \P_m U}{\kY} \leq m \, \norm{\P_m U}{\kY} 
\end{equation}
for $m \in \N$.  Let $q>0$, $\tau \geq 0$ and $\ell \in \N_0$.  Since
$A$ is normal, $\lvert \Q A \rvert^{\ell} \exp( \tau \lvert \Q A
\rvert^{1/q})$ is a well-defined, generally unbounded and densely
defined operator on $\kY$.  We may thus introduce the abstract
\emph{Gevrey space}
\begin{equation}
  \label{e.Ytauell}
  \kY_{\tau,\ell} =  \kY_{\tau,\ell,q} = D( \lvert \Q A \rvert^{\ell} \,
  \exp( \tau \lvert \Q A \rvert^{1/q}))
\end{equation}
equipped with the inner product
\begin{multline}
  \label{e.YtauellProduct}
  \langle U_1, U_2 \rangle_{\kY_{\tau,\ell}}
   = \langle \P U_1, \P U_2 \rangle_{\kY} \\ + 
     \langle 
       \lvert \Q A \rvert^{\ell} \, 
       \exp( \tau \vert \Q A \rvert^{1/q}) \, \Q U_1,
       \lvert \Q A \rvert^{\ell} \, 
       \exp( \tau \vert \Q A \rvert^{1/q}) \, \Q U_2
     \rangle_{\kY} \,.
\end{multline}
Gevrey-smooth functions $U \in \kY_{\tau,\ell }$ are exponentially
well approximated by their Galerkin projections $\P_m U$.  Indeed,
setting $\Q_m = \id - \P_m$,
\begin{equation}
  \norm{\Q_m U}\kY 
  \leq m^{-\ell} \, \exp(-\tau m^{1/q}) \,
       \norm{U}{\kY_{\tau,\ell}} \,.
 \label{e.qm-estimate}
\end{equation}
Moreover, this definition of the norm ensures that
\begin{equation}
  \label{e.normA}
  \norm{A}{\kY_{\tau,\ell+1} \to \kY_{\tau,\ell}} \leq 1 
  \quad \text{and} \quad
  \norm{U}{\kY_{\tau,\ell}} \leq \norm{U}{\kY_{\tau,\ell+1}}
\end{equation}
for all $U \in \kY_{\tau,\ell+1}$.  For convenience, we define
$\kY_\ell \equiv \kY_{0,\ell}$.  We can then state the following
lemma which will be needed later.

\begin{lemma} \label{Le:GmHp}
Let $A$ satisfy \textup{(A)}. Then for $\sigma>\tau$ and $p \in \N_0$,
\begin{equation}
  \norm{A^p}{\kY_{\sigma,\ell} \to \kY_{\tau,\ell}} 
  \leq \max \biggl\{ 
              1, \biggl(\frac{pq}{\e(\sigma-\tau)}\biggr)^{pq}
            \biggr\}
  \label{e.gmhp}
\end{equation}
\end{lemma}

\begin{proof} 
For $p=0$, there is nothing to prove, hence let $p>0$.  For fixed $U
\in \kY_{\sigma,\ell}$,
\begin{equation}
  \norm[2]{A^p U}{\kY_{\tau,\ell}}
  = \norm[2]{\P U}{\kY} 
   + \norm[2]{\lvert \Q A \rvert^p \, \e^{(\tau-\sigma) \, 
              \lvert \Q A \rvert^{1/q}} \, \lvert \Q A \rvert^\ell \,
              \e^{\sigma \lvert\Q A \rvert^{1/q}} \, \Q U}{\kY} \,.
  \label{e.gmhp-proof}
\end{equation}
The function $f(\lambda) = \lvert \lambda \rvert^{p} \,
\e^{(\tau-\sigma) |\lambda|^{1/q}}$ is non-negative and has a global
maximum at $\lambda_{*} = (pq/(\sigma-\tau))^{q}$.  Replacing the
corresponding term in \eqref{e.gmhp-proof} by its maximum value, we
obtain \eqref{e.gmhp}.
\end{proof}

\subsection{Differentiability of the semiflow}
\label{s.differentiability}

To obtain a flow of the evolution equation which has higher order time
derivatives, as required in Section~\ref{s.bea}, we need more specific
assumptions on the regularity of $B$ on the scale of Hilbert spaces
defined above.

We use the following convention to denote derivatives.  Given any
Hilbert space $\kX$, open set $\kD \subset \kX$, and map $Z \colon \kD
\to \R$, we write $\D Z(U)$ to denote the derivative of $Z$ at $U \in
\kD$ as an element of $\kX^*$, and by $\nabla Z(U)$ the canonical
representation of $\D Z(U)$ by an element of $\kX$.  In other words,
$\D Z(U) W = \langle \nabla Z(U), W \rangle$, where $\langle \cdot,
\cdot \rangle$ denotes the inner product on $\kX$.  

For Hilbert spaces $\kX$ and $\kZ$, and $j \in \N_0$, we write
$\kE^j(\kY,\kX)$ to denote the vector space of $j$-multilinear bounded
mappings from $\kY$ to $\kX$; we set $\kE^j(\kX) \equiv
\kE^j(\kX,\kX)$.  Moreover, when $\kU \subset \kX$ is open and $k\in
\N$, we write $\kCb^k(\kU,\kZ)$ to denote the set of $k$ times
continuously differentiable functions $F \colon \kU \to \kZ$ whose
derivatives $\D^i F$ are bounded as maps from $\kU$ to
$\kE^i(\kX,\kZ)$ and extend to the boundary of $\kU$.  When $\kU$ is
not open, but has non-empty interior, we define $\kCb^k(\kU,\kZ)=
\kCb^k(\interior\kU,\kZ)$, where $\interior \kS$ denotes the interior
of a set $\kS$.

Finally, for Hilbert spaces $\kX$, $\kY$, and $\kZ$, and open subsets
$\kU \subset \kX$, $\kV \subset \kY$, and $\kW \subset \kZ$, we write
\[
  F \in \kCb^{(\underline{m},n)} (\kU \times \kV; \kW)
\]
to denote a continuous, bounded function $F \colon \kU \times \kV \to
\kW$ whose partial Fr\'echet derivatives $\D_X^i \D_Y^j F(X,Y)$ exist,
are bounded, and are such that the maps
\[
  (X, Y, X_1,\ldots, X_{i}) 
  \mapsto \D_X^i \D_Y^j  F(X,Y)(X_1, \ldots, X_{i}) 
\]
are continuous from $\kU \times \kV \times \kX^{i}$ into $\kE^j(\kY,
\kZ)$ for $i = 0, \dots, m$ and $j = 0,\dots, n$, and extend
continuously to the boundary.  If $\kU$ or $\kV$ are not open but have
non-empty interior, we again define $\kCb^{(\underline{m},n)} (\kU
\times \kV; \kW) = \kCb^{(\underline{m},n)} (\interior\kU \times
\interior\kV; \kW)$.

Given $\delta>0$ and a family of open sets $\kD_\ell \subset \kY_\ell$
for $\ell = 0, \ldots, L$ for $L \in \N$, we define the sets
\begin{equation}
  \label{e.Dkdelta}
  \kD^{\delta}_\ell 
  = \bigcup_{U \in \kD_\ell} \kB_{\delta}^{\kY_\ell}(U)
\end{equation}
analogously to the set $\kD^\delta$ above. 

We assume that the sets $\kD_\ell$ are nested, i.e., $\kD_{\ell+1}
\subset \kD_\ell$, for $\ell=0,\ldots, L-1$.  Then, by construction,
we also have $\kD^\delta_{\ell+1} \subset \kD^\delta_\ell$ for
$\ell=0,\ldots, L-1$.  For example, the family $\kD_k =
\interior\kB_R^{\kY_k}(U^0)$ is nested for every $U^0 \in \kY_L$ and
$R>0$.

We now make the following assumption on the nonlinearity of our
semilinear evolution equation.

\begin{itemize}
\item[(B1)] For $\delta>0$ fixed as in (B0), there exist $K \in \N_0$,
$N \in \N$ with $N > K+1$, and a nested sequence of $\kY_k$-bounded
and open sets $\kD_k$, such that $B \in \kCb^{N-k}(\kD_k^\delta,
\kY_k)$ for $k = 0, \dots, K$.
\end{itemize}

We denote the bounds of the maps $B \colon \kD_k^\delta \to \kY_k$ and
their derivatives by constants $M_k$, $M_k'$, etc., for $k = 0, \dots,
K$.  In addition to the domains $\kD_0, \dots, \kD_K$ defined in (B1),
we also need a domain $\kD_{K+1}$ on the next higher scale rung
$\kY_{K+1}$ which may be any $\kY_{K+1}$-bounded, open, and nested
subset of $\kD_K$, and we define
\begin{equation}
  \label{e.Rk}
  R_{K+1} = \sup_{U \in \kD_{K+1}^\delta} \norm{U}{\kY_{K+1}} \,.
\end{equation}

We can then quote the following theorem on the uniform regularity of
the flow \cite[Theorem~2.6 and Remark~2.8]{OliverWulffRK}.

\begin{theorem}[Regularity of semiflow] \label{c.local-wp-unif} 
Assume \textup{(A)} and \textup{(B1)}.  Then there exists $T_*>0$ such
that the semiflow $(U,t) \mapsto \Phi^t(U)$ of \eqref{e.pde} satisfies
\begin{equation}
  \label{e.T-diff-gen-unif}
  \Phi \in \bigcap_{\substack{j+k \leq N \\ \ell \leq k \leq K+1}}
  \kCb^{(\underline j, \ell)} (\kD_{K+1} \times (0,T_*); 
  \kY_{k-\ell}) \,.  
\end{equation}
Moreover, $\Phi$ maps $\kD_{K+1}\times [0,T_*]$ into $\kD_K^\delta$.
The bounds on $\Phi$ and $T_*$ depend only on the bounds afforded by
\textup{(B1)}, \eqref{e.Rk}, on $\omega$, and on $\delta$.
\end{theorem}
 
\begin{remark}
In \cite{OliverWulffRK} we chose to shrink domains $\kD$ in the range
of the flow map to
\[
  \kD^{-\delta} 
  = \{ U \in \kD \colon \dist(U, \partial \kD) \geq \delta\}
\]
rather than work with extended domains $\kD^\delta$ in the argument of
the flow map as we do here.  Since $\kD^{\epsilon-\delta} \subset
(\kD^\epsilon)^{-\delta}$ for $\epsilon>\delta>0$, the formulation in
\cite{OliverWulffRK} implies the version stated here; working with
extended domains is more convenient for the purposes of this paper as
the extension preserves star-shapedness which is required in
Section~\ref{s.hamiltonian} below.
\end{remark}

\begin{remark} \label{r.superposition} 
The precise form of assumption (B1) is motivated by the typical case
where $B$ is a superposition operator of a function $f \colon D
\subset \R^d\to \R^m$ and $\kY_\ell$ is related to the standard
Sobolev space $\kH_\ell = \kH_\ell(I; \R^d)$, where $I =[a,b] \subset
\R$.  Then, if $f$ is $(N+1)$ times continuously differentiable on
some open set $D \subset \R^d$, it is $N$ times differentiable as a
map from the open set $\kD$ of $\kH_1$ to $\kH_1$.  We note that $u
\in \kD$ ensures that $u(x) \in D$ pointwise.  Moreover, $f$ is
$(N-k)$ times differentiable from $\kD \cap \kH_k$ to $\kH_k$---see
\cite[Theorem 2.12]{OliverWulffRK} and also \cite[Remark 7.4]{OWW04}.
\end{remark}

For the results of Section~\ref{s.bea}, we need regularity of the flow
on a space of Gevrey-regular functions as well.  Hence, we assume the
following.
 
\begin{itemize}
\item[(B2)] There exist $\tau>0$, $q>0$, $L \geq 0$, and a
$\kY_{\tau,L}$-bounded open set $\kD_{\tau,L} \subset \kD_{K+1}$ such
that $B \in \kCb^{2}(\kD_{\tau,L}^\delta, \kY_{\tau,L})$
\end{itemize}

We note that $\kY_{\tau,L} \subset \kY_{K+1}$ due to
Lemma~\ref{Le:GmHp}. In the following, $\kD_{\tau,L+1}$ refers to an
arbitrary $\kY_{\tau,L+1}$-bounded open subset of $\kD_{\tau,L}$.  We
note that under assumption (B2), Theorem~\ref{c.local-wp-unif} applies
with $\kY_{\tau,L}$ in place of $\kY$.


\subsection{Hamiltonian structures on Hilbert spaces}
\label{s.hamiltonian}

Our main result, Theorem~\ref{t.GenMain} below, requires that
\eqref{e.pde} is Hamiltonian, i.e., that there exists a symplectic
structure operator $\JJ$ and a Hamiltonian $H \colon \kD \to \R$ such
that
\begin{equation}
  \label{e.HamPDE}
  \partial_t U = AU + B(U) = \JJ \nabla H(U) \,.
\end{equation}
In addition to (A) and (B0), we assume the following.

\begin{itemize}
\item[(H0)] The symplectic structure operator $\JJ$ is a closed,
skew-symmetric, densely defined, and bijective linear operator on
$\kY$.

\item[(H1)] $A$ is skew-symmetric and $\JJ^{-1}A$ is bounded and
self-adjoint on $\kY$.

\item[(H2)] For every $U \in \kD^\delta$, the operator $\JJ^{-1} \D
B(U)$ is self-adjoint on $\kY$.

\item[(H3)] $\kD$ is star-shaped.

\end{itemize}
Formally, the function $H \colon\kD^\delta \to \R$ is an invariant of
the motion.
 
Recall that a subset $\kS$ of a linear space is star-shaped if there
exists $U_* \in \kS$ such that for every $W \in \kS$, the line segment
$U_*W$ is contained in $\kS$.  We then say that $\kS$ is star-shaped
with respect to $U_*$.  We remark that if $\kD$ is star-shaped with
respect to $U^* \in \kD$, then $\kD^\delta$ is star-shaped with
respect to $U^*$ as well.  Moreover, by the closed graph theorem, (H0)
implies that $\JJ$ is invertible with $\JJ^{-1} \in \kE(\kY)$.  This
implies, in particular, that $\JJ^{-1} \D B(U)$ is a bounded operator
on $\kY$ for every $U \in \kD^\delta$.

An operator $A$ is skew if $A^* = -A$ and $D(A)=D(A^*)$.  This implies
that $\spec(A) \subset \i \R$ and that, by Stone's theorem, $A$
generates a unitary $\kC^0$-group on $\kY$; see, e.g.\
\cite{ReedSimon}.  If $A=A_{\mathrm s} + A_{\mathrm b}$ where
$A_{\mathrm s}$ is skew and $A_{\mathrm b}$ is bounded, we can
redefine $B$ as $B+ A_{\mathrm b}$ and $A$ as $A_{\mathrm s}$ to
satisfy (H1).  This situation is typical for semilinear wave
equations, see Section~\ref{ss.swe}.
 
Further, by conditions (A), (H0), and (H1),
\begin{equation}
  \label{e.JAcommute}
  \JJ^{-1} A = (\JJ^{-1}A)^*
  = A^* (\JJ^{-1})^* = -A (-\JJ^{-1}) = A \JJ^{-1} \,.
\end{equation}
Hence, $A$ and $\JJ^{-1}$ commute, which also implies the following.

\begin{lemma} \label{Le:JPm}  
Assume \textup{(A)}, \textup{(H0)}, and \textup{(H1)}.  Then $\JJ^{-1}
\P_m = \P_m \JJ^{-1}$ for all $m \in \N_0$.
\end{lemma}
 
\begin{proof}
By \eqref{e.JAcommute}, $A$ and $\JJ^{-1}$ commute, and so do $F(A)$
and $\JJ^{-1}$ for all analytic functions $F$.  Approximating
characteristic functions $\chi_{\Lambda}$ of measurable sets $\Lambda
\subset \C$ by analytic functions, we see that $\chi_{\Lambda}(A)$ and
$\JJ^{-1}$ also commute \cite{ReedSimon}.  With $\Lambda =
\kB^{\C}_m(0)$, this implies that $\chi_{\Lambda}(A) = \P_m$ commutes
with $\JJ^{-1}$.
\end{proof}

The existence of a Hamiltonian $H$ is then guaranteed by the following
integrability lemma.

\begin{lemma} \label{l.H2}
Assume \textup{(A)}, \textup{(B0)} and \textup{(H0--3)}.  Then there
exists an analytic bounded Hamiltonian $H \colon \kD^\delta \to \R$
for the evolution equation \eqref{e.pde}.
\end{lemma}

\begin{proof}
We seek a Hamiltonian of the form
\begin{equation}
  \label{e.EnergyPDE}
  H(U) = \frac12 \, \langle U, \JJ^{-1} A \, U \rangle + V(U) \,.
\end{equation}
Due to (H1), the quadratic part of the Hamiltonian is well-defined and
possesses the properties claimed.

To proceed, we use that $\kD^\delta$ is star-shaped and fix $U^0$ such
that $\kD$ and $\kD^\delta$ are star-shaped with respect to $U^0$.  We
set
\begin{equation}
  V(U) = \int_0^1 
  \langle \JJ^{-1} B (t\, U + (1-t) \, U^0), 
  U-U^0 \rangle \, \d t \,,
  \label{e.v2}
\end{equation}
so that, for $W \in \kY$, 
\begin{align}
  \langle \nabla V(U), W \rangle
  & = \int_0^1 
      \langle \JJ^{-1} B (t\, U + (1-t) \, U^0), W \rangle  \, \d t
    \notag \\
  & \quad
    + t \int_0^1
      \langle \JJ^{-1}\D B(t\,U + (1-t) \, U^0)W,U-U^0 \rangle \, \d t
    \notag \\
  & = \int_0^1 
      \frac{\d}{\d t} \langle t \, \JJ^{-1} B (t\, U + (1-t) \, U^0), W
      \rangle \, \d t \,,
  \label{e.HamComput}
\end{align}
where the last equality is due to the self-adjointness of $\JJ^{-1} \D
B(U)$.  Then, by the fundamental theorem of calculus, $B(U) = \JJ
\nabla V(U)$.  Further, \eqref{e.v2} shows that analyticity of $B$
implies analyticity of $V$, and uniform bounds on $B$ imply
corresponding uniform bounds on $V$.
\end{proof}

For bounds on the modified Hamiltonian in Section~\ref{s.bea} we will
need (H3) on at least two scale rungs, so that, for simplicity, we
assume the following.

\begin{itemize}
\item[(H4)] Each $\kD_k$ is star-shaped for $k = 0, \dots, K+1$.
\end{itemize}

In the next section we introduce concrete function spaces and
superposition operators on these spaces in order to verify that our
main examples, the nonlinear Schr\"odinger equation and the semilinear
wave equation, fit into our abstract framework.


\subsection{Spaces of analytic functions}
\label{ss.analyticFcts}

We denote the Fourier coefficients of a function $u \in \kL_2(\S^1;
\C^d)$ on the circle $\S^1 \simeq \R/(2\pi\Z)$ by $\hat{u}_k$, so that
\begin{equation}
  u(x) = \frac{1}{\sqrt{2\pi}} 
         \sum_{k\in\Z} \hat{u}_k \, {\e}^{{\i}kx} \,.
  \label{e.fourier}
\end{equation}
Let $\kG_{\tau,\ell} \equiv \kG_{\tau,\ell}(\S^1; \C^d)$ denote the
Hilbert space of analytic functions $u \in \kL_2(\S^1; \C^d)$ for
which
\[
  \norm[2]{u}{\kG_{\tau,\ell}}
  \equiv \langle u , u \rangle_{\kG_{\tau,\ell}} < \infty \,,
\]
where the inner product is given by
\begin{equation}
  \label{e.inner-product}
  \langle u, v \rangle_{\kG_{\tau,\ell}} 
  = \sum_{|k|\leq 1} \langle \hat{u}_k \, \hat{v}_k \rangle_{\C^d} 
    + \sum_{|k|>1} k^{2\ell} \, \e^{2\tau|k|} \, 
      \langle \hat{u}_k \, \hat{v}_k \rangle_{\C^d} \,.
\end{equation}
It can be shown that $\kG_{\tau,\ell}(\S^1; \C^d)$ contains all real
analytic functions whose radius of analyticity is at least $\tau$. In
particular, functions in $\kG_{\tau,\ell}(\S^1; \C^d)$ can be
differentiated infinitely often.  This follows from Lemma
\ref{Le:GmHp} with $\kY=\kL_2(\S^1; \C^d)$ and $A= \partial_x$.  We
write $\kH_\ell \equiv \kG_{0,\ell}$ to denote the usual Sobolev space
of functions whose weak derivatives up to order $\ell$ are
square-integrable.

The additional index $\ell$ in $\kG_{\tau,\ell}$ is important because
of the following.

\begin{lemma}[{\cite[Lemma 1]{Ferrari98}}] \label{l.algebra}
The space $\kG_{\tau,\ell}(\S^1; \C)$ is a topological algebra for
every $\tau \geq 0$ and $\ell > 1/2$. Specifically, there exists a
constant $c=c(\ell)$ such that for every $u, v \in
\kG_{\tau,\ell}(\S^1; \C)$ the product $uv \in \kG_{\tau,\ell}(\S^1;
\C)$ with
\begin{equation}\label{eq:uvGm}
  \norm{uv}{\kG_{\tau,\ell}(\S^1;\C)}
  \leq c \, \norm{u}{\kG_{\tau,\ell}(\S^1;\C)} \, 
            \norm{v}{\kG_{\tau,\ell}(\S^1;\C)} \,.
\end{equation}
\end{lemma}

To treat general nonlinear potentials, we need to consider
superposition operators $f \colon \kG_{\tau,\ell} \to \kG_{\tau,\ell}$
of analytic functions.  The following lemma is a minor adaptation of
results proved in \cite{Ferrari98,Matthies01}.

\begin{lemma}\label{l.nonlinearityG^rl}
If $f \colon \C^d \to \C^d$ is entire then $f$ is also entire as a
function from $\kG_{\tau,\ell}(\S^1; \C^d)$ to itself for every $\tau
\geq 0$ and $\ell > 1/2$.  If $f$ is analytic on $\kB_{r}(u_0) \subset
\C^d$ where $u_0 \in \C^d$ then $f$ is analytic from $\kD_{\tau,\ell}
\equiv \kB_{R}(u_0)\subset \kG_{\tau,\ell}(\S^1;\C^d)$ to
$\kG_{\tau,\ell}(\S^1; \C^d)$ for any $R<r/c$, where $c=c(\ell)$ is
the constant from Lemma \ref{l.algebra}.  Moreover, $f \in
\kCb^2(\kD_{\tau,\ell};\kG_{\tau,\ell})$.
\end{lemma}
 
\begin{proof} 
We prove the result for $d=1$; it generalizes to $d>1$ if
multi-indices are used.  Let $f$ be entire and let
\begin{equation}
  \label{e.fTaylor}
  f(z) = \sum_{n=0}^\infty \, a_n \, (z-u_0)^n \,
\end{equation}
be the Taylor series  of $f$ around $u_0 \in \C$. Let $\phi \colon \R
\to \R$ be its majorization
\[
  \phi(s) = \sum_{n=0}^\infty \, \lvert a_n \rvert \, s^{n} \,.
\]
By applying the algebra inequality \eqref{eq:uvGm} to each term of the
power series expansion \eqref{e.fTaylor} of $f(u)$, we see that the
series converges for every $u \in \kG_{\tau,\ell}$ provided $\tau \geq
0$ and $\ell > 1/2$, and that
\begin{equation}
  \norm{f(u)}{\kG_{\tau,\ell}}
  \leq c^{-1} \, \phi \bigl( c \, \norm{u - u_0}{\kG_{\tau,\ell}} \bigr) 
 + |a_0|(\sqrt{2\pi} - c^{-1}) \,,
 \label{e.majorant}
\end{equation}
where $c$ is as in Lemma~\ref{l.algebra}, see \cite{Ferrari98}.  In
other words, $f$ is entire on $\kG_{\tau,\ell}(\S^1)$.

When $f$ has only a finite radius of analyticity, we argue as follows.
Assume that $\lvert f(z) \rvert \leq M$ on $\kB_r^\C(u_0)$.  Then, by
Cauchy's estimate,
\[
  \lvert a_n \rvert \leq \frac{M}{r^n} \,.
\]
Consequently, the majorant $\phi$ is bounded on any $\kB_\rho^\C(0)$
with $\rho<r$ with uniform bound
\[
  \mu = \frac{M}{1 - \rho/r} \,.
\]
Due to \eqref{e.majorant}, the superposition operator $f$ is then
analytic and bounded by
\[
  M_{\text{spp}} = \mu/c + \lvert a_0 \rvert \, (\sqrt{2\pi} - c^{-1})  
\]
as a map from a ball $\kD_{\tau,\ell} = \kB_{R}(u_0)$ of radius
$R=\rho/c$ around $u_0 \in \kG_{\tau,\ell}$ into $\kG_{\tau,\ell}$;
similarly, we see that $f \in \kCb^2(\kD_{\tau,\ell};
\kG_{\tau,\ell})$.
\end{proof}

 
\subsection{Functional setting for the semilinear wave equation}
\label{ss.swe}

Consider the semilinear wave equation
\begin{equation}
  \label{e.swe}
  \partial_{tt}u = \partial_{xx} u - V'(u)
\end{equation} 
on $\S^1$.  Its Hamiltonian can be written
\[
  H(u,v) = \int_{\S^1}
  \biggl[
    \frac12 \, v^2 + \frac12 \, (\partial_x u)^2 + V(u)
  \biggr] \, \d x 
\]
where $v = \partial_t u$.  We write $U=(u,v)^T$ and set $\kY \equiv
\kH_1(\S^1;\R) \times \kL_2(\S^1;\R)$ so that the Hamiltonian is
well-defined on $\kY$.  For $U=(u,v) \in \kY$ let $\P_0 U = (\p_0 u,
\p_0 v)$ where, for $u \in \kL_2(\S^1;\R)$, we define $\p_0 u =
\hat{u}_0$ and let $\Q_0 = \id - \P_0$.  Setting
\[
  \tilde A =
  \begin{pmatrix}
  0 & \id \\ \partial^2_x & 0
  \end{pmatrix} \,,  
\]
we then define
\begin{equation}
  \label{e.AB}
  A = \Q_0 \tilde A \,, \qquad
  B(U) =
  \begin{pmatrix}
    0 \\ -V'(u)
  \end{pmatrix}
  + \P_0 \tilde A U \,,
\end{equation}
and the symplectic structure operator $\JJ$ via
\begin{equation}
  \label{e.J}
  \langle \JJ^{-1} U_1, U_2 \rangle_\kY 
  = \int_{\S^1} (u_1 v_2 - u_2 v_1) \, \d x
\end{equation}
for all $U_1 = (u_1,  v_1)^T, U_2 = (u_2, v_2)^T \in \kY$.
 
Since the Laplacian is diagonal in the Fourier representation
\eqref{e.fourier} with eigenvalues $-k^2$ for $k \in \Z$, the
eigenvalue problem for $A$ separates into $2\times 2$ eigenvalue
problems on each Fourier mode, and $\spec A = \i \Z \setminus \{0\}$.
Clearly, $A$ is skew-symmetric on $\kY$ if $\kH_1 = \kG_{0,1}$ is
endowed with the inner product \eqref{e.inner-product}.  Note that
$\P_0 \tilde A$ has a Jordan block and is hence included with the
nonlinearity $B$.  Thus, with
\[
  \kY_{\tau,\ell} = \kG_{\tau,\ell+1}(\S^1;\R) \times 
                    \kG_{\tau,\ell}(\S^1;\R) \,,
\]
$B(U)$ from \eqref{e.AB} satisfies (B2) with $q=1$.
 
The symplectic structure operator $\JJ$ defined by \eqref{e.J} is an
unbounded operator on $\kY_{\tau,\ell}$ with domain
$\kY_{\tau,\ell+1}$.  It is possible, though not necessary for
anything which follows, to compute $\JJ^{-1}$ explicitly.  Namely,
\eqref{e.J} reads
\begin{equation}
  \langle \JJ^{-1} U_1, U_2 \rangle_\kY 
  = \langle (\JJ^{-1} U_1)_u, u_2 \rangle_{\kH_1} + 
    \langle (\JJ^{-1} U_1)_v, v_2 \rangle_{\kL_2}
  = \int_{\S^1} (u_1 v_2 - u_2 v_1) \, \d x \,.
  \label{e.j-definition-wave}
\end{equation}
The definition of the inner product \eqref{e.inner-product} implies
\[
  \langle (\JJ^{-1} U_1)_u, u_2\rangle_{\kH_1} 
  = \langle (\p_0  -\partial_x^2) (\JJ^{-1} U_1)_u, u_2
    \rangle_{\kL_2} \,, 
\]
so that \eqref{e.j-definition-wave} splits into 
\[
  (\p_0 - \partial_x^2) (\JJ^{-1} U_1)_u = -v_1
  \qquad \text{and} \qquad
  (\JJ^{-1} U_1)_v = u_1 \,.
\]
We conclude that
\[
  \JJ^{-1} = 
  \begin{pmatrix}
    0 & -(\p_0-\partial^2_x)^{-1}\\
    1 & 0 
  \end{pmatrix} \,.
\]
If the potential $V \colon D \to \R$ is analytic on an open set $D
\subset \R$, then, by Lemma~\ref{l.nonlinearityG^rl}, $B$ is analytic
from $\kB_{R}^{\kY{\tau,\ell}}(U^0)$ to $\kY_{\tau,\ell}$ for any
$\tau,\ell \geq 0$ and $U^0 = (u^0,v^0) \in \kY_{\tau,\ell}$ with
$u^0$ independent of $x$, provided $\kB_r^{\R}(u^0) \subset D$ with $R
< r/c(\ell)$ as in Lemma~\ref{l.nonlinearityG^rl}.  In this setting,
all the above assumptions are satisfied.
 

\subsection{Functional setting for the nonlinear Schr\"odinger
  equation}
\label{ss.nse}

Consider the nonlinear Schr\"odinger equation
\begin{equation}
  \label{e.nse}
  \i \, \partial_t u 
  = - \partial_{xx} u + \partial_{\overline{u}} V(u,\overline{u})
\end{equation}
on the circle $\S^1$, where $V(u,\overline{u})$ is analytic in $\Re
u$ and $\Im u$.  Setting $U \equiv u$, we can write
\begin{equation}
  \label{e.nlsdefs}
  A = \i \, \partial^2_x \,,  \quad 
  B(U) = -\i \, \partial_{\overline{u}} V(u,\overline{u}) \,. 
\end{equation}
Similar to \eqref{e.J}, we have
\begin{equation}
  \label{e.J_nls}
 \langle \JJ^{-1} u_1, u_2 \rangle_\kY 
  = \int \Re(\i u_1 \bar{u}_2) \, \d x
\end{equation}
with $\kY = \kH_1(\S^1,\C)$.  Therefore, as for the semilinear wave
equation in Section~\ref{ss.swe}, we see that $\JJ^{-1} \colon
\kH_2\to \kL_2$ and that (H0--1) hold.  The first term of the
Hamiltonian
\begin{equation}
  \label{e.HamiltonianNSE}
  H(U) = \frac{1}{2} \int_{\S^1} 
         \bigl( 
           \lvert \partial_x u \rvert^2 + V(u,\overline{u})
         \bigr) \, \d x
\end{equation}
is then well-defined for all $u \in \kY$.  The Laplacian is diagonal
in the Fourier representation \eqref{e.fourier} with eigenvalues
$-k^2$.  Hence, $ \spec A = \{ -\i k^2 \colon k \in \Z\}$ so that $A$
generates a unitary group on $\kL_2(\S^1;\C)$ and, more generally, on
every $\kG_{\tau,\ell}$ with $\ell \in \N_0$ and $\tau \geq 0$.

To continue, we identify $\R^2 \simeq \C$ so that $\kY
=\kH_1(\S^1,\R^2)$.  If the potential $V \colon D \subset \R^2 \to \R$
is analytic as a function of $(q,p) \equiv (\Re u, \Im u)$, then, by
Lemma~\ref{l.nonlinearityG^rl}, the nonlinearity $B(U)$ defined in
\eqref{e.nlsdefs} is analytic as map from a ball in
$\kG_{\tau,\ell}(\S^1;\R^2)$ to itself for every $\tau \geq 0$ and
$\ell > 1/2$ .  The construction of the domain hierarchy works as in
Section~\ref{ss.swe}, so that (B0--2) hold with $\kY_{\tau,\ell} =
\kG_{\tau,2\ell+1}(\S^1;\R^2)$ where $\ell \in \N_0$ and $q=2$.

We remark that if we were to write out the nonlinear Schr\"odinger
equation in real coordinates with $U = (\Re u, \Im u)$, the structure
operator $\JJ$ would be the canonical symplectic matrix on $\R^2$ for
the space $\kL_2(\S^1;\R^2)$.

An example of a nonlinear Schr\"odinger equation in Fourier space
which is defined on a more complicated set than a ball can be found in
Remark~\ref{r.nonlocal} below.


\section{A-stable Runge--Kutta methods on Hilbert spaces}
\label{s.rk}

In this section, we introduce a class of A-stable Runge--Kutta methods
which are well-defined when applied to the semilinear PDE
\eqref{e.pde} under assumptions (A) and (B0), and review some
regularity and convergence results for those methods from
\cite{OliverWulffRK}.  In most of this section, we need not assume
that \eqref{e.pde} is Hamiltonian.

Applying an $s$-stage Runge--Kutta method of the form \eqref{eq:RK} to
the semilinear evolution equation \eqref{e.pde}, we obtain
\begin{subequations} \label{eq:RK}
\begin{align}
  W & = {U}^0 \, \1 + h \, {\a} \, F(W)  \,,
   \label{eq:RKStagesIter} \\
 \Psi^h(U^0) & = U^0 + h \, {\b}^T \, F( W)  \,. 
   \label{eq:RKUpdate}
\end{align}
\end{subequations}
For $U\in \kY$, we write
\[
  \1 \, U = 
  \begin{pmatrix}
    U \\ \vdots \\ U
  \end{pmatrix} \in \kY^s \,,
  \quad
  W = 
  \begin{pmatrix}
    W^1 \\ \vdots \\ W^s
  \end{pmatrix} \,,
  \quad
  {B}(W) = 
  \begin{pmatrix}
    B(W^1) \\ \vdots \\ B(W^s)
  \end{pmatrix} \,,
\]
where $W^1, \dots, W^s$ are the stages of the Runge--Kutta method,
\[
  (\a W)^i = \sum_{j=1}^s \a_{ij} \, W^j \,,
  \qquad
  \b^T W  = \sum_{j=1}^s \b_{j} \, W^j \,,
\]
and $A$ acts diagonally on the stages, i.e., $({A}W)^i = A W^i$ for
$i=1,\dots,s$.  

The scheme \eqref{eq:RK} can be written in a more suitable
form, required later, namely
\begin{subequations}
\begin{equation}
  \label{eq:newRKStagesIter}
  W = \Pi(W; U, h) 
    \equiv (\id - h  {\a} {A} )^{-1} \, 
    (\1 U + h  {\a}  {B}(W)) 
\end{equation}
and  
\begin{align}
  \Psi^h(U)  = \mS(hA) U + h {\b}^T \, (\id - h \a A)^{-1} \, B(W(U,h)) \,,
  \label{eq:NewRKUpdate}
\end{align}
\end{subequations}
where $S$ is the so-called \emph{stability function}
\begin{equation}
  \label{e.SzRK}
  \mS(z) = 1 + z \b^T \, (\id - z \a)^{-1} \, \1 \,.
\end{equation}

We now make a number of assumptions on the method and its interaction
with the linear operator $A$.  First, we assume that the method is
A-stable.  Setting $\C^- = \{ z \in \C \colon \Re z \leq 0\}$, the
conditions are as follows.
\begin{itemize}
\item[(RK1)] The stability function \eqref{e.SzRK} is bounded with
$\vert \mS(z) \rvert \leq 1$ for all $z \in \C^-$.

\item[(RK2)] The $s \times s$ matrices $\id - z \a$ are invertible for
all $z \in \C^-$.
\end{itemize}
We also require two further conditions.
\begin{itemize}
\item[(RK3)] The matrix $\a$ is invertible.
\item[(RK4)] The method is symplectic.
\end{itemize}
Recall that the flow map of a Hamiltonian system is a symplectic map,
i.e., $\Phi^t$ satisfies
\begin{equation}
  \label{e.symplecticity}
  (\D_U \Phi^t(U))^T \, \JJ^{-1} \, \D_U \Phi^t(U) 
  = \JJ^{-1}
\end{equation}
for all $U$ and $t$ for which this relation is well defined \cite{MR}.
A numerical one-step method is called symplectic if, when applied to a
Hamiltonian system, its time-$h$ map $\Psi^h$ is symplectic.

It is known that a Runge--Kutta method of the form \eqref{eq:RK} is
symplectic when its coefficients satisfy
\[
  \b_i \, \a_{ij} + \b_j \, \a_{ji} - \b_i \, \b_j = 0 
\]
for $i,j = 1, \dots, s$; see, for example, \cite{Sanz-SernaCalvo}.  The
simplest example of a symplectic Runge--Kutta method is the
\emph{implicit midpoint rule}, given by
\[
 \Psi^h(U)= U + h \, F \Bigl(\frac{U+ \Psi^h(U)}2 \Bigr) \,,
\]
which, equivalently, can be written in the form of a general
Runge--Kutta scheme \eqref{eq:RK} with $s=1$, $\a_{11}= \frac{1}{2}$,
and $\b_1 = 1$.  This is an example of a Gauss--Legendre Runge--Kutta
method; Gauss--Legendre Runge--Kutta methods satisfy conditions
\textup{(RK1--4)}; see \cite[Lemma 3.6]{OliverWulffRK} for condition
(RK1--3) and, e.g., \cite{Sanz-SernaCalvo} for (RK4).

In the following, we also need to refer to a set of key estimates on
the linear operators which appear in the formulation
\eqref{eq:newRKStagesIter} and \eqref{eq:NewRKUpdate} of the
Runge--Kutta method, namely
\begin{subequations} \label{e.RKShA}
\begin{gather}
  \norm{(\id - h\a A)^{-1}}{\kY^s \to \kY^s} \leq \Lambda \,,  
  \label{e.RKShA.a} \\
  \norm{h \a A (\id - h \a A)^{-1}}{\kY^s \to \kY^s} \leq 1 + \Lambda \,,
  \label{e.RKShA.b} \\ 
  \norm{\mS(hA)}{\kY^s \to \kY^s} \leq 1 + \sigma \, h
  \label{e.RKShA.c} 
\end{gather}
\end{subequations}
for all $h \in [0,h_*]$ and constants $\Lambda \geq 1$ and $\sigma\geq
0$.  These estimates naturally hold true on each rung of our hierarchy
of spaces.  For proofs, see \cite[Section~3.2]{OliverWulffRK}.
 
A-stable Runge--Kutta methods have the remarkable property that their
time-$h$ map is of the same regularity class as the flow of the
evolution equation stated in Theorem~\ref{c.local-wp-unif}.  We state
this result as an abbreviated version of \cite[Theorem~3.15 and
Remark~3.17]{OliverWulffRK}.
 
\begin{theorem}[Regularity of numerical method]
\label{t.h-diff-unif}
Assume \textup{(A)}, \textup{(B1)}, and \textup{(RK1--3)}.  Then there
exists $h_{*}>0$ such that the components $W^j$ of the stage vector
$W(U,h)$ and numerical method $\Psi(U,h)=\Psi^h(U)$ are of class
\eqref{e.T-diff-gen-unif} with $T_*$ there replaced by $h_*$ here.
Moreover, $\Psi$ and $W^j$ map into $\kD_K^\delta$.  The bounds on
$W$, $\Psi$ and $h_*$ only depend on the bounds afforded by
\textup{(B1)} and \eqref{e.Rk}, on the coefficients of the method, on
the constants afforded by \eqref{e.RKShA} and on $\delta$.
\end{theorem}

\begin{remark} \label{rem:1IMPR}
Even when $B \colon \kY \to \kY$ is analytic, the numerical time-$h$
map $\Psi^h(U)$ is generally not analytic in $h$ unless $A$ is
bounded.  Take, for example, the linear Schr\"odinger equation, i.e.,
equation \eqref{e.nse} with $B \equiv 0$, discretized by the implicit
mid point rule.  Then
\[
  h \mapsto \mS(hA) \, e_k 
  = (\id+ \tfrac12 h A) \, (\id - \tfrac12 h A)^{-1} \, e_k 
  = \frac{1 + \tfrac12 h k^2 \i}{1 - \tfrac12 h k^2 \i} \, e_k
\]
has radius of analyticity $\frac{2}{k}$, where $e_k$ is the $k$-th
Galerkin mode of $A$ as described in Section~\ref{ss.nse}.  Therefore,
if the Fourier expansion of $U$ does not terminate finitely, then
$\Psi^h(U) =\mS(hA) U$ cannot be analytic in $h$.  This argument
applies to any A-stable Runge--Kutta method: Since $\lvert \mS(z)
\rvert \leq 1$ for all $z \in \i\R$ by assumption (RK1), the stability
function is a rational polynomial $\mS(z) = \mP(z)/\mQ(z)$ with $\deg
\mQ \geq \deg \mP$.  Hence, $\deg \mQ \geq 1$ so that $\mS(z)$ has at
least one pole $z_0$.  The radius of analyticity of $\mS(z)$ around
$0$ is $r_0 = \lvert z_0 \rvert$, so that $h \mapsto \mS(hA)e_k$
cannot be analytic outside a ball around $h=0$ of radius $r_0/k$.  As
for the implicit midpoint rule, this implies that $\Psi^h(U)$ is not
analytic in $h$ unless the Fourier expansion of $U$ is finite.
\end{remark}

Thus, while differentiability in $h$ can be obtained by stepping down
on a scale of Hilbert spaces, analyticity can only be obtained by
projecting onto a subspace on which the vector field is bounded.  This
will become necessary in Section~\ref{ssec:embedFlow} where
analyticity is essential for obtaining exponential error estimates.


\section[Exponentially accurate embeddings of
time-semidiscretizations]{Exponentially accurate Hamiltonian
embeddings\\of time-semidiscretizations}
\label{s.bea}

We are now ready to state and prove our main result on approximate
embeddings of symplectic time-discretizations of Hamiltonian evolution
equations into flows.

\subsection{Statement of the main result}
\label{ss.mainResult}

In the following, we write $j=\lfloor r \rfloor$ to denote the largest
integer $j\leq r$, and $j=\lceil r \rceil$ to denote the smallest
integer $j \geq r$.

\begin{theorem}[Main theorem] \label{t.GenMain}
Assume that the semilinear Hamiltonian evolution equation
\eqref{e.HamPDE} with energy \eqref{e.EnergyPDE} satisfies
\textup{(A)}, \textup{(B0--2)}, and \textup{(H0--4)}.  Apply a
symplectic Runge--Kutta method of order $p \geq 1$ and step size $h$
which satisfies \textup{(RK1--4)} to \eqref{e.HamPDE}.  Assume further
that  
\begin{equation*}
  K+1 \geq P \equiv \lceil p(q+1)^2/q + q \rceil + 1
\end{equation*}
with $K$ from \textup{(B1)} and $q$ from \textup{(B2)}.  Then there
exists $h_*>0$ and a modified energy $\tilde{H} \colon
\kD_{1}^{\delta/2} \times [0,h_*] \to \R$ which is analytic in $U$ for
each $h \in [0,h_*]$ and satisfies
\begin{subequations}
\begin{equation}
  \label{e.HtildeHClose}
  \sup_{U \in \kD_{P}}
  \lvert \tilde{H}(U,h) - H(U,h) \rvert
   = O(h^p) 
\end{equation}
such that $\JJ \nabla \tilde H$ generates a modified flow $\tilde \Phi
\colon \kD_1 \times [0,h_*] \to \kD_1^{\delta/4}$.  The numerical
method is approximately embedded into the modified flow with
exponentially small error in the sense that there is $c_*>0$ such that
\begin{equation}
  \sup_{U \in \kD_{\tau,L+1}}  
  \norm{\Psi^h(U) - \tilde \Phi^h (U)}{\kY_1}
  \leq c_{\tilde\Phi} \, \e^{-c_* h^{-\frac{1}{1+q}}} \,.
  \label{e.hilbert-embedding2}
\end{equation}
The modified energy is also approximately conserved by the numerical method:
\begin{equation}
  \label{e.onestepHamExpEst}
  \sup_{U \in \kD_{\tau,L+1}} 
    \lvert \tilde{H}(\Psi^h(U),h) - \tilde{H}(U,h) \rvert
  \leq c_{\tilde H} \, \e^{-c_* h^{-\frac{1}{1+q}}} \,.
\end{equation}
\end{subequations}
\end{theorem}
 
For the semilinear wave equation, $q=1$ (see Section~\ref{ss.swe}), so
that the exponents in \eqref{e.hilbert-embedding2} and
\eqref{e.onestepHamExpEst} scale like $h^{-1/2}$.  For the nonlinear
Schr\"odinger equations, $q=2$ (see Section~\ref{ss.nse}), so that the
exponents in \eqref{e.hilbert-embedding2} and
\eqref{e.onestepHamExpEst} scale like $h^{-1/3}$.  Note that, due to
(B2), $\kD_{\tau,L+1} \subset \kD_{K+1} \subset \kD_{P}$, so that the
supremum in \eqref{e.HtildeHClose} can be taken, in particular, over
$\kD_{\tau,L+1}$.

\begin{remark} \label{r.expAllTimes} For ODEs, estimate
\eqref{e.onestepHamExpEst} holds with $q=0$ and implies approximate
conservation of the energy $H$ over exponentially long times so long
as the numerical trajectory remains bounded.  For PDEs this conclusion
does not hold, because solutions of semilinear PDEs and their
discretizations are not generally Gevrey regular over long times,
while Gevrey regularity is needed for the embedding estimate
\eqref{e.hilbert-embedding2}, cf.\ Section~\ref{s.LowerEst}.
\end{remark}
 
\begin{remark}
When the evolution equation \eqref{e.HamPDE} is linear, e.g., a linear
wave equation or linear Schr\"odinger equation, its Hamiltonian is
conserved exactly since symplectic Runge--Kutta methods conserve
quadratic invariants \cite{HaLu2002}.
\end{remark}
 
\begin{remark} \label{r.FiniteTime} 
Our result implies conservation of the modified energy with
exponentially small error over finite times under slightly stronger
conditions.  We assume that there is a triple of Gevrey spaces as
follows.
\begin{itemize}
\item[(B3)] There are $\tau>0$, $q>0$, $L \geq 0$, and a sequence of
nested $\kY_{\tau,L+k}$-bounded and open sets $\kD_{\tau,L+k}$ such
that $B \in \kCb^{3-k}(\kD_{\tau,L+k}^\delta, \kY_{\tau,L+k})$ for $k
=0, 1, 2$.
\end{itemize} 
Let $\kD_{\tau,L+3} \subset \kD_{\tau,L+2}$ be an open and bounded
subset of $\kY_{\tau,L+3}$.  Fix $T>0$ and $\delta>\varepsilon>0$.
Then for any $U^0$ with
\begin{equation}
  \label{e.beaSolnCond}
  \{\Phi^t(U^0) \colon t  \in [0,T] \} 
  \subset \kD_{\tau,L+3}
\end{equation}
and for any $h\in[0,h_*]$, the convergence theorem \cite[Theorem
3.20]{OliverWulffRK} with $p \equiv 1$, with $\kY_{\tau,L+1}$ in place
of $\kY$, ensures that there is $h_*>0$ such that for any $h\in
(0,h_*]$, the discrete trajectory $U^j=(\Psi^h)^j(U^0)$ is
$O(h)$-close to the flow in the $\kY_{\tau,L+1}$ norm and hence
satisfies $U^j \in \kD_{\tau,L+1}^\varepsilon$ so long as $0 \leq j
\leq \lfloor T/h \rfloor$ and $h>0$ is sufficiently small.
Theorem~\ref{t.GenMain} with $\kD_{\tau,L+1}$ replaced by
$\kD_{\tau,L+1}^\varepsilon$, $\kD_j$ replaced by $\kD_j^\varepsilon$,
and $\delta$ reduced to $\delta-\epsilon$ then implies approximate
conservation of the modified energy $\tilde H$ with an error
\[
  h^{-1} \, O(\e^{-c_* h^{-\frac{1}{1+q}}}) 
  = O(\e^{-\beta h^{-\frac{1}{1+q}}})
\]
for any $\beta \in (0,c_*)$ with order constants uniform over all
$U^0$ satisfying \eqref{e.beaSolnCond}.
\end{remark}

The remainder of the section is devoted to the proof of
Theorem~\ref{t.GenMain}, where claims
(\ref{e.HtildeHClose}--\ref{e.onestepHamExpEst}) correspond to
Lemma~\ref{l.H-TildeH}, Lemma \ref{l.gevrey-embedding} and
Lemma~\ref{l.gevrey-hamiltonian}, respectively.

Lemma~\ref{l.gevrey-embedding} generalizes the well known embedding
result for ODEs, stated as Theorem~\ref{t.embedding} below, to the
Hilbert space setting.  Theorem~\ref{t.embedding} is not directly
applicable to PDEs because the formal expansion in $h$ of the
numerical method contains powers of the unbounded operator $A$.  We
thus resort to the following construction, which is also used in
\cite{MaSch2001}: In Section~\ref{s.truncation}, we truncate the
evolution equation \eqref{e.HamPDE} to the subspace $\P_m \kY$.  Then,
in Section~\ref{ssec:embedFlow}, we obtain an embedding result on this
subspace and choose an optimal cut-off $m$ as function of $h$ to
obtain the embedding result in the Hilbert space setting.  Finally, in
Section~\ref{ssec:formalBEA}, we prove estimate
\eqref{e.HtildeHClose}.


\subsection{Galerkin truncation}
\label{s.truncation}

For given $m \in \N$, we define a truncated Hamiltonian evolution
equation by restricting the Hamiltonian phase space to the subspace
$\P_m \kY$.  Since $\nabla H \vert_{\P_m \kY} = \P_m \nabla H$ and
$\JJ^{-1}$ leaves $\P_m \kY$ invariant by Lemma~\ref{Le:JPm}, the
corresponding restricted evolution equation reads
\[
  \dot u_m = \JJ \, \P_m \nabla H(u_m) \,.
\]
Thus, setting $f_m = \P_m F$ and $B_m=\P_m B$, we
can write
\begin{equation} 
  \label{e.slwe-m}
  \dot u_m \equiv f_m(u_m) = A u_m + B_m(u_m) \,.
\end{equation}
We denote the flow of the projected system on $\P_m \kY$ by
$\phi^t_m$.  For convenience, we set $\Phi^t_m = \phi^t_m \circ \P_m$.
Similarly, let $w_m$ denote the stage vector, $w_m^j$ for $j=1, \dots,
s$ its components, and $\psi_m^h$ denote the numerical time-$h$ map
obtained by applying an $s$-stage Runge--Kutta method to the projected
semilinear evolution equation \eqref{e.slwe-m}, and abbreviate $W_m^j
= w_m^j \circ \P_m$ and $\Psi_m^h = \psi_m^h \circ \P_m$.

In \cite{OliverWulffGalerkin}, we proved that all of the maps
above---the truncated flow $\Phi^t_m$, the components of the stage
vector $W_m^j$, and the time-$h$ map $\Psi_m^h$ of the truncated
system---are of the same class \eqref{e.T-diff-gen-unif} as the exact
flow $\Phi^t$ with $m$-independent bounds.  The precise statement is
as follows.
  
\begin{theorem}[Regularity of flow and numerical method of projected
system] \label{t.wm_psim-uniform}
Assume \textup{(A)}, \textup{(B1)}, and \textup{(RK1--3)}.  Then there
are positive $T_*$, $h_*$, and $m_*$ such that for every $m \geq m_*$
the flow $\Phi^t_m$ is of class \eqref{e.T-diff-gen-unif}, and the
components of the numerical stage vector $W_m^j$ and the numerical
time-$h$ map $\Psi^h_m$ are of the same class, but with $T_*$ replaced
by $h_*$, with bounds which are independent of $m \geq m_*$.
Moreover, $\Phi^t_m$, $W_m^j$, and $\Psi_m^h$ map $\kD_{K+1}$ into
$\kD_K^\delta$.
\end{theorem}

Note that $\Phi^t_m$ and $\Psi^h_m$ are analytic in $t$ resp.\ in $h$
so long as $B$ is analytic on $\kY$.  However, the radius of
analyticity is generally not uniform in $m$.

Next, we present an exponential error bound for the projection error
of the numerical scheme; this is necessary for obtaining an
exponentially small embedding error in Section~\ref{ssec:embedFlow}
below.

\begin{lemma}[Exponential projection error estimate for the numerical scheme]
\label{l.projection-error} 
Assume that the semilinear evolution equation \eqref{e.pde} satisfies
conditions \textup{(A)} and \textup{(B0--2)}.  Let, as before,
$\Psi^h$ and $\Psi_m^h$ denote a single step of a Runge--Kutta method
subject to \textup{(RK1--3)} applied to the full and the projected
semilinear evolution equation, \eqref{e.pde} and \eqref{e.slwe-m},
respectively.  Then there are positive constants $h_*$, $m_*$, and
$c_\Psi$ such that for all $m \geq m_*$, $h \in [0,h_*]$ and $U \in
\kD_{\tau,L+1}$,
\begin{equation}
  \norm{\Psi^h(U) - \Psi_m^h (U)}{\kY_1}
  \leq c_\Psi \, m^{-L} \, \e^{-\tau m^{1/q}} \,.
  \label{e.projection-error.c} 
\end{equation}
\end{lemma}

\begin{proof}
By Theorems~\ref{t.h-diff-unif} and~\ref{t.wm_psim-uniform} applied
with $\kY_{\tau,L}$ in place of $\kY$ and $K=0$, there exist $h_*>0$
and $m_* > 0$ such that
\[
  \Psi^h, \Psi^h_m, W^j, W_m^j  \in 
  \kCb (\kD_{\tau,L+1} \times [0,h_*];
    \kY_{\tau,L+1} \cap \kD_{\tau, L}^\delta) 
\]
for $j=1,\ldots, s$ with bounds which are uniform in $h \in (0,h_*]$
and $m \geq m_*$.

We first estimate the difference of the stages vectors $W(U) -
W_m(U)$, noting that
\begin{align*}
  W(U) 
  = & (\id - h {\a} A)^{-1} \, 
      \bigl( \1 U  + h \a B(W(U)) \bigr) \,, \\
  W_m(U) 
  = & (\id - h {\a} A)^{-1} \, 
      \bigl( 
        \P_m \1 U  + h \, \P_m \a B(W_m(U)) 
      \bigr) \,. 
\end{align*}
Taking the difference of both expressions, we obtain
\begin{equation}
  W(U) - W_m(U) 
    = (\id - h {\a} A)^{-1} \, 
      \bigl(
        \Q_m \1 U + h \, \a \,
        \bigl[
          B(W(U)) - \P_m B(W_m( U)) 
        \bigr]
       \bigr) \,.
  \label{e.w-wm}
\end{equation}
Setting 
\begin{equation*}
  \norm{\b}{} = \sum_{i=1}^s \lvert \b_i \rvert
  \quad \text{and} \quad
  \norm{\a}{} 
  = \max_{i=1,\ldots, s} \sum_{j=1}^s \lvert \a_{ij} \rvert \,,
\end{equation*}
and using estimate \eqref{e.RKShA.a}, we obtain
\begin{multline*}
  \norm{h \, \a(\id - h {\a} A)^{-1} (B(W(U)) - 
        \P_m B(W_m(U)))}{\kY^s} \\
  \leq h \, \Lambda \, \norm{\a}{} \,
        \norm{B(W(U)) - \P_m B(W_m(U))}{\kY^s} \,.
\end{multline*} 
Using the triangle inequality and the mean value theorem, we further
estimate
\begin{align}
  \norm{B(W(U)) & - \P_m B(W_m(U))}{\kY^s} 
    \notag \\
  & \leq \norm{\Q_m B(W(U))}{\kY^s} + 
         \norm{\P_m B(W(U)) - \P_m B(W_m(U))}{\kY^s} 
    \notag \\
  & \leq \norm{\Q_m B(W(U))}{\kY^s} 
         + M_0' \, \norm{W(U) - W_m(U)}{\kY^s} \,.
\label{e.bw-wm}
\end{align} 
Taking the $\kY^s$ norm of \eqref{e.w-wm}, using \eqref{e.RKShA.a},
and inserting \eqref{e.bw-wm}, we obtain
\begin{align*}
  \norm{W(U) - W_m(U)}{\kY^s}
  & \leq s \, \Lambda \, \norm{\Q_m U}{\kY} 
         +  h \,\norm{\a}{} \, \Lambda \, \norm{\Q_m B(W(U))}{\kY^s}
         \notag \\
  & \quad +
         h \, M_0' \, \norm{\a}{} \, \Lambda \, 
         \norm{W(U) - W_m(U)}{\kY^s} \,.
\end{align*}
This proves that, for $h_* < {1}/{(M_0' \norm{\a}{} \Lambda)}$, 
\[
  \norm{W(U) - W_m(U)}{\kY^s}
  \leq s \, \Lambda \, \frac{\norm{\Q_m U}{\kY}}%
                       {1 - h_*M_0'\Lambda \norm{\a}{}}  
     + h_* \, \norm{\a}{} \, \Lambda \, 
       \frac{\norm{\Q_m B(W(U))}{\kY^s}}%
            {1 - h_* M_0'\Lambda \norm{\a}{}} \,.
\]
We apply \eqref{e.qm-estimate} to the first term on the right and note
that, again by \eqref{e.qm-estimate},
\begin{equation}
  \norm{\Q_m B(W(U))}{\kY^s} 
  \leq s \, m^{-L} \, \e^{-\tau m^{1/q}} \, M_{\tau,L} \,,
  \label{e.qmbw}
\end{equation}
where $M_{\tau,L}$ is the bound for the norm of $B \colon
\kD_{\tau,L}^\delta \to \kY_{\tau,L}$ afforded by assumption (B2).
This establishes that there exists a constant $c_W$ such that
 \begin{equation}
  \norm{W(U) - W_m(U)}{\kY^s}
  \leq c_W \, m^{-L} \, \e^{-\tau m^{1/q}}
  \label{e.projection-error.a}  
\end{equation}
for all $m \geq m_*$ and $h \in [0,h_*]$.  (As in
\cite{OliverWulffGalerkin}, where we considered the case $\tau=0$, we
could obtain a stage vector error bound in the $\kY_1^s$-norm by
applying $A$ onto \eqref{e.w-wm} and using \eqref{e.RKShA.b}, but this
is not necessary for what follows.)

To estimate the difference between the Runge--Kutta updates, we write
them in the form \eqref{eq:NewRKUpdate} such that the respective right
hand sides are (uniformly) bounded operators,
\begin{align*}
  \Psi^h (U)
  & = \mS(hA) \, U + h \, \b^T \, (\id - h \a A)^{-1} \,
      B(W(U)) \,, \\
  \Psi_m^h (U)
  & = \mS(hA) \, \P_m U + h \, \b^T \, (\id - h \a A)^{-1} \,
      \P_m B(W_m(U)) \,.
\end{align*}
Then,
\begin{multline*}
  \Psi^h (U) - \Psi_m^h (U)
  = \mS(hA) \, \Q_m U 
    + h \, \b^T \, (\id - h \a A)^{-1} \,
      \bigl[
        B(W(U)) - \P_m B(W_m(U))
      \bigr] \,.
\end{multline*}
Inserting \eqref{e.qmbw} and \eqref{e.projection-error.a} back into
\eqref{e.bw-wm}, we also find that
\begin{equation}
  \norm{B(W(U)) - \P_m B(W_m(U))}{\kY^s}
  \leq c_B \, m^{-L} \, \e^{-\tau m^{1/q}} 
  \label{e.projection-error.b}   
\end{equation}   
for some constant $c_B>0$.  We note that \eqref{e.YtauellProduct} and
\eqref{e.RKShA.b} imply
\[
  \norm{h \a  (\id - h \a A)^{-1}}{\kY^s \to \kY_1^s} 
  \leq 2 + \Lambda \,.
\]
This and the invertibility of $\a$, assumption (RK3), then yield
\begin{align*}
  & \norm{\Psi^h (U) - \Psi_m^h (U)}{\kY_1} \\
  & \leq (1 + \sigma h) \, \norm{\Q_m U}{\kY_1} + 
       \norm{\b}{} \, \norm{\a^{-1}}{} \,
       \norm{h \a(\id - h {\a} A)^{-1}(B(W(U)) - 
             \P_m B(W_m(U)))}{\kY_1^s} \\
  & \leq (1 + \sigma h) \, \norm{\Q_m U}{\kY_1} + 
       (2+\Lambda) \, \norm{\b}{} \, \norm{\a^{-1}}{} \,
       \norm{ B(W(U)) - \P_m B(W_m(U))}{\kY^s} \,. 
\end{align*}
Inequality \eqref{e.projection-error.c} is now a consequence of
\eqref{e.qm-estimate} and \eqref{e.projection-error.b}.
\end{proof}


\subsection{Approximate embedding of semidiscretizations into a flow}
\label{ssec:embedFlow}

We first review an approximate embedding result for Runge--Kutta
discretizations of ODEs and then show how to extend it to PDEs.
Consider the autonomous ODE
\begin{equation}
  \label{eq:ode}
  \dot{y} = f(y)
\end{equation}
defined on the closed ball $\kB^{\C^m}_{r}(y^0)$.  We write $\phi^t$
to denote the flow of \eqref{eq:ode} and $\psi^{h}$ denote the
time-$h$ map of an $s$-stage Runge--Kutta method of the form
\eqref{eq:RK} applied to \eqref{eq:ode}.  When $f$ is analytic, it is
known that $\psi^{h}$ can be expanded in a converging power series in
$h$ on a smaller ball, so that we can write
\begin{equation}
  \psi^{h}(y) 
  = y + \sum_{j=1}^\infty h^j \, g^j(y) \,.
  \label{e.psi-expansion}
\end{equation}
Specifically, as shown in \cite[Theorem IX.7.2]{HaLu2002} (with $2R$
there replaced by $r$ here) via Cauchy estimates,
\eqref{e.psi-expansion} holds true on $\kB^{\C^m}_{r/2}(y^0)\times [0,
r/({4 \lVert \a \rVert M}))$.  Moreover, the numerical time-$h$ map
$\psi^h$ can be embedded into the flow of a modified vector field up
to an exponentially small error.  A general form of this result was
proved in \cite{BeGi1994} with specific proofs for the class of
Runge--Kutta schemes we consider in \cite{HaLu2002,Re1999}, see also
\cite{LeRe04}.  We state the result as follows.
 
\begin{theorem} \label{t.embedding}
In the setting introduced above, there are positive constants $\eta$,
$c_1$, and $c_2$ which depend only on the method such that for every
$r>0$ and $M>0$ such that
\[
  \norm{f(y)}{} \leq M 
  \quad \mbox{for} \quad y \in \kB^{\C^m}_{r}(y^0)
\]
and every $h \in [0,\eta r/M]$ there exists a modified differential
equation $\dot{y} = \tilde{f}(y)$, defined on $\kB^{\C^m}_{r/4}(y^0)$,
whose flow $\tilde\phi^t$ satisfies $\tilde\phi^t(y^0) \in
\kB^{\C^m}_{r/4}(y^0)$ for at least $0 \leq t \leq h$ and which
satisfies the estimate
\[
  \norm{\psi^h(y^0) - \tilde\phi^h (y^0)}{} 
  \leq h \, c_1 \, M \, \e^{-\tfrac{c_2r}{hM}}_{\vphantom 1} \,.
\]
\end{theorem}

The proof shall not be repeated in detail here.  But we note, for
later reference, that the modified vector field is constructed as a
power series in $h$,
\begin{equation}
  \label{eq:tildeFSeries}
  \tilde{f}^n (y;h) 
  = f(y) + \sum_{j=p}^{n-1} h^j \, f^{j+1}(y) \,,
\end{equation}
where $p$ is the order of the numerical method.  Its exponential map
is then expanded in powers of $h$ and matched term by term with the
expansion of the numerical time-$h$ map \eqref{e.psi-expansion}.  This
yields a recursive expression for the coefficient vector fields, 
\begin{equation}
  \label{eq:jModf}
  f^j(y) = g^j(y) - \sum_{i=2}^j \frac1{i!} 
          \sum_{k_1+\ldots + k_i = j}
          \left( \D_{k_1} \cdots \D_{k_{i-1}} f^{k_i} \right)(y) \,, 
\end{equation}
for $j \geq 2$ where $k_i \geq 1$ for all $i$, see
\cite[Section~3.1]{BeGi1994}.  We write $\D_i g(y) = \D g(y) f^i(y)$
as short hand notation for the Lie derivative with respect to the
$i$th coefficient vector field, cf.\ \cite[Lemma IX.7.3]{HaLu2002}.
The proof of Theorem~\ref{t.embedding} proceeds by carefully
estimating the growth of the $f^j$, noting that the optimal truncation
is achieved when $n=n(h) = \lfloor c_2 r/(hM) \rfloor$.  When
referring to the optimally truncated vector field, we write
$\tilde{f}_h$ or just $\tilde{f}$.

In addition, we need the following estimate which guarantees
consistency of the truncation.  It is a slight generalization of
results proved in \cite{HaLu2002, Re1999}.
 
\begin{lemma} \label{l.fh-f} 
In the notation of Theorem~\textup{\ref{t.embedding}}, for every  $a
\geq 1$ there exists a constant $c_3=c_3(a)$ such that for every $h
\in [0,\eta r/M]$, 
\[
  \norm{\tilde{f}^a(y) - \tilde f(y)}{} 
  \leq c_3 \, r^{-a} \, M^{a+1} \, h^a
  \quad \text{for} \quad 
  y \in \kB^{\C^m}_{r/4}(y^0) \,. 
\]
\end{lemma}

\begin{proof}
It is known, see \cite{BeGi1994} for general numerical one-step
methods and \cite{HaLu2002,Re1999} for the Runge--Kutta methods
considered here, that there exist positive constants $c_4$ and
$c_5\leq 1/(c_2\e)$ which depend only on the method such that
\begin{equation}
  \label{e.fj}
  \norm{f^j(y)}{} 
  \leq c_4 M \, \biggl(
                  \frac{c_5 Mj}r
                \biggr)^{j-1} 
  \quad\mbox{for}\quad
  y \in \kB^{\C^m}_{r/4}(y^0) \,.
\end{equation}
Applying this estimate to \eqref{eq:tildeFSeries} and using that $n
\leq c_2 r/(hM)$ and therefore $h \leq c_2r/(nM)$, we find that
\begin{align}
  \norm{\tilde f^a(y) -\tilde f(y)}{}
  & \leq h^a \sum_{j=a}^{n-1} h^{j-a} \, \norm{f^{j+1} (y)}{}
         \notag \\
  & \leq h^a \sum_{j=a}^{n-1} h^{j-a} \, c_4 \, M \,
         \biggl( \frac{c_5 M(j+1)}{r} \biggr)^j  
         \notag \\
  & \leq c_4 \, M \, h^a \sum_{j=a}^{n-1}
         \biggl( \frac{c_2 r}{nM} \biggr)^{j-a} \,
         \biggl( \frac{c_5 M(j+1)}{r} \biggr)^j  
         \notag \\
  & \leq c_4 \, \e^{a+1} \, M \,
         \biggl(\frac{c_2 h M}{r} \biggr)^a \, \sum_{j=a}^{n-1} 
            \biggl( \frac{j+1}{n} \biggr)^{j-a} \,
             \frac{(j+1)^{a}}{\e^{j+1}}  
         \notag\\
  & \leq c_4 \, \e^{a+1} \, M \,
         \biggl(\frac{c_2 h M}{r} \biggr)^a \, a! \,, 
  \label{e.fh-f}
\end{align}
where, in the last inequality, we have bounded the first factor inside
the sum by $1$ and noted that $j^a \, \e^{-j}$ is decreasing for
$j\geq a$, so that
\[
  \sum_{j=a}^{n-1} \frac{(j+1)^a}{\e^{j+1}} 
  \leq \int_a^n x^a \, \e^{-x} \, \d x
  \leq \int_0^\infty x^a \, \e^{-x} \, \d x
  = a! \,.
\]
This completes the proof.
\end{proof}

Since $\tilde f^{p} = f$, we note, setting $a=p$, that
Lemma~\ref{l.fh-f} provides a bound on $\norm{f(y) - \tilde f(y)}{}$.
Hence, by the triangle inequality, for every $h \leq \eta r/M$ there
is a method-dependent constant $c_{\tilde f}$ such that
\begin{equation}
  \label{e.normtildef}
  \norm{\tilde f (y)}{} \leq M \, c_{\tilde f}
  \quad \text{for} \quad 
  y \in \kB^{\C^m}_{r/4}(y^0) \,. 
\end{equation}

We now apply Theorem~\ref{t.embedding} to the sequence of truncated
problems \eqref{e.slwe-m} where, for each $m$, we work on the space
$\P_m \kY$ endowed with the $\kY_1$-norm.  (See
Remark~\ref{r.WhyWorkInY1} for an explanation of why we work with the
$\kY_1$ rather than the $\kY$-norm.)

Let $\kY_1^\C = \kY_1 + \i\kY_1$ be the complexification of $\kY_1$.
We now choose $R_1>0$ such that $\kD_{1}^\delta \subset
\kB^{\kY_1}_{R_1}(0)$.  Then, by construction, $f_m$ is analytic on
$\kD^\C \cap \P_m\kB_{R_1}^{\kY_1^\C}(0)$ and satisfies the estimate
\begin{align}
  \norm{f_m(u_m)}{\kY^\C_{1}} 
  & \leq \norm{A u_m}{\kY^\C_{1}} 
         + \norm{B_m (u_m)}{\kY^\C_{1}} 
    \notag \\
  & \leq m R_1
         + m \, \sup_{U \in \kD^\C} \norm{\P_m B(U)}{\kY^\C}
    \leq c_F \, m 
  \label{e.fm-estimate}
\end{align}
with $c_F = R_1 + M_0$.

Setting $M = c_F \, m$, Theorem~\ref{t.embedding} asserts that the
numerical time-$h$ map can be embedded into a modified flow up an
error which is exponentially small in the step size $h$, albeit not
uniformly in $m$.  If, however, we make the stronger assumption that
the initial data lies some Gevrey space $\kY_{\tau,L+1}$ with
$\tau>0$, Lemma~\ref{l.projection-error} asserts that the numerical
solution of the full semilinear evolution equation \eqref{e.pde}
remains exponentially close in the spectral cutoff $m$ to the
numerical solution of the projected system.  Thus, we can carefully
choose $m=m(h)$ to balance the projection error and the embedding
error to obtain an embedding result on the Gevrey space which is still
exponential in $h$, but at a lesser rate.  This is done in the next
lemma, where we also show that the result can be formulated not only
on balls as in Theorem~\ref{t.embedding}, but also on more general
$m$-independent subdomains of $\kY_1$ as needed in the proof of
Theorem~\ref{t.GenMain}.

We denote the coefficients of the power series expansion of $\psi_m^h$
by $g_m^j$, the expansion coefficients of the modified vector
field---defined via \eqref{eq:jModf}---by $f_m^j$, and seek an
optimally truncated modified vector field of the form
\begin{equation}
  \tilde f_m^n(u_m;h)
  = f_m(u_m) + \sum_{j=p}^{n-1} h^j \, f_m^{j+1}(u_m) \,.
  \label{e.tildefm}
\end{equation}
 
\begin{lemma}[Embedding lemma for Gevrey class data]
\label{l.gevrey-embedding}
Assume that the semilinear evolution equation \eqref{e.pde} satisfies
conditions \textup{(A)} and \textup{(B0--2)}.  Let, as before,
$\Psi^h$ denote a single step of a Runge--Kutta method subject to
\textup{(RK1--3)} applied to the semilinear evolution equation
\eqref{e.pde}.  Then there exists $h_* >0$ such that the choices
\begin{subequations}
\begin{equation}
  \label{e.m}
  m(h) = \Bigl( \frac{\chi}{\tau h} \Bigr)^{\tfrac{q}{1+q}}
  \qquad \text{and} \qquad
  n(h) = \bigg\lfloor
             \tau^{\tfrac{q}{1+q}} \,
             \Bigl( \frac{\chi}{h} \Bigr)^{\tfrac{1}{1+q}} 
             \bigg\rfloor 
\end{equation}
with $\chi = c_2 \delta/(4 c_F)$ ensure that the modified vector field
\begin{equation}
  \label{e.modVF}
  \tilde F(U;h) \equiv \tilde f_{m(h)}^{n(h)}(\P_m U; h) \,,
\end{equation}
where $\tilde f_m^n$ is given by \eqref{e.tildefm}, has the following
properties.
\begin{enumerate}
\item $\tilde F \colon \kD_1^{\delta/2}\to \kY_1$  is analytic
for every fixed $h \in [0,h_*]$ with bound
\begin{equation}
  \label{e.boundtildeF}
  \norm{\tilde F}{\kCb(\kD_1^{\delta/2}, \kY_1)} 
  \leq c_{\tilde F} \, m(h)
\end{equation}
for some $c_{\tilde F}>0$ independent of $h$.

\item $\tilde F$ generates a modified flow $\tilde \Phi \colon \kD_1
\times [0,h] \to \kD_1^{\delta/4}$.

\item There exists a constant $c_{\tilde\Phi}$ such that for every $U
\in \kD_{\tau, L+1}$ and $h \in [0,h_*]$,
\begin{equation}
  \norm{\Psi^h(U) - \tilde \Phi^h (U)}{\kY_1}
  \leq c_{\tilde\Phi} \, \e^{-c_* h^{-\frac{1}{1+q}}} 
  \label{e.hilbert-embedding}
\end{equation}
with $c_* =\tau^{\frac{q}{1+q}} \, \chi^{\frac{1}{q+1}}$.

\item For each $a \in \N$ there is a constant $c_a\geq 0$ such that,
with $\tilde F^a_m \equiv \tilde f_{m }^{a} \circ \P_{m}$ and
$m=m(h)$, we have
\begin{equation}
  \label{e.tildefa-tildeF}
  \norm{\tilde F  - \tilde F_{m}^{a}}{\kCb(\kD_1^{\delta/2}, \kY_1)} 
  \leq c_{a} \, h^a \, m^{a+1} \,.
\end{equation} 
\end{enumerate}
\end{subequations}
\end{lemma} 

\begin{proof} 
Set $r = \delta/4$.  As $\kD_1^\delta$ is a bounded subset of $\kY_1$,
there exists $m_*$ such that, due to \eqref{e.qm-estimate},
$\norm{\Q_m U}{\kY} \leq m^{-1} \, \norm{U}{\kY_1} \leq \delta/2 - r$
and therefore $\P_m U \in \kD^{\delta-r}$ for every $U \in
\kD_1^{\delta/2}$ and $m \geq m_*$.  In particular, for any such $U$
and $m$, the ball
\[
  \kB_r^{\P_m \kY_1^\C} (\P_m U) 
  = \{ u \in \P_m \kY^\C \colon \norm{u - \P_m U}{\kY_1^\C} \leq r \}
\]
is contained in $\kD^\C \cap \P_m \kB_{R_1}^{\kY_1^\C} (0)$.  Then
estimate \eqref{e.fm-estimate} holds true and we can apply
Theorem~\ref{t.embedding} with $M = c_Fm$ and $y^0 = \P_m U$ on this
ball.  This theorem asserts that for every $h \in [0, \eta r/(c_Fm)]$,
the modified vector field $\tilde f_m = \tilde f_m^{n(h)}$ with $n(h)
= \lfloor c_2r/(hc_Fm) \rfloor$ is defined on $\kB_{r/4}^{\P_m
\kY_1^\C}(\P_m U)$ and analytic as map from $\kB_{r/4}^{\P_m
\kY_1^\C}(\P_m U)$ to $\kY_1^\C$.  Its flow $\tilde \phi_m^t$
satisfies $\tilde \phi^t_m(\P_m U) \in \kB_{r/4}^{\P_m\kY_1^\C}(\P_m
U)$ for at least $0 \leq t \leq h$, and
\begin{equation}
  \norm{\psi_m^h(\P_m U) - \tilde \phi^h_m(\P_m U)}{ \kY_1^\C }
  \leq h \, c_1 \, c_F \, m \, 
       \e^{-\tfrac{c_2r}{hc_Fm}}_{\vphantom 1} \,.
  \label{e.exp-proj-estimate}
\end{equation}
By construction, $\tilde{F}_m = \tilde{f}_m \circ \P_m$ is analytic as
map from $\kD_1^{\delta/2}$ to $\kY_1$ with flow map $\tilde\Phi_m =
\tilde\phi_m\circ \P_m + \Q_m$ which is analytic as map from $\kD_1$
to $\kD_1^{\delta/4}$ for any choice of $m\geq m_*$ and $t \in [0,h]$.
Estimates \eqref{e.boundtildeF} and \eqref{e.tildefa-tildeF} then
follow directly from \eqref{e.normtildef} and Lemma~\ref{l.fh-f},
respectively.
  
Our next step is to estimate the difference between the solution to
the modified projected equation and the numerical solution of the full
semilinear evolution equation \eqref{e.pde}.  We split the error into
the projection and the embedding error, the first of which is
controlled by Lemma~\ref{l.projection-error} and the second is
controlled by \eqref{e.exp-proj-estimate}.  By
Theorems~\ref{t.h-diff-unif} and~\ref{t.wm_psim-uniform},
respectively, with $K=0$ and $\kY$ replaced by $\kY_{\tau,L}$, there
exist $h_*>0$ and (a possibly increased choice of) $m_*$ such that,
for $h \in [0,h_*]$ and $m \geq m_*$, $\Psi^h$ and $\Psi^h_m$ are
continuous maps from $\kD_{\tau,L+1}$ to $\kY_{\tau,L+1}$ and such
that the exponential projection error estimate asserted by
Lemma~\ref{l.projection-error} holds.  Possibly increasing $m_* \geq
\eta r/(c_F h_*)$, we ensure that both the embedding error estimate
\eqref{e.exp-proj-estimate} and the truncation error estimate
\eqref{e.projection-error.c} hold true for every $m \geq m_*$ and $h
\in [0,\eta r/(c_Fm)]$.

Since $\kD_{\tau,L+1} \subset \kD_{1}$, splitting the total error into
a projection error component and the embedding error on the subspace
$\P_m \kY$, we obtain that
\begin{align*}
  \norm{\Psi^h (U) - \tilde \Phi_m^h (U)}{\kY_1}
  & \leq \norm{\Psi^h(U) - \psi_m^h (\P_m U)}{\kY_1} + 
         \norm{\Q_mU}{\kY_1} \\
  & 
         + \norm{\psi_m^h (\P_m U) - \tilde \phi_m^h (\P_m U)}{\kY_1}
         \notag \\
  & \leq (1+c_\Psi) \, m^{-\ell} \, \e^{-\tau m^{1/q}} \,
         + h \, c_1 \, c_F \, m \, 
           \e^{-\tfrac{c_2r}{hc_Fm}}_{\vphantom 1} 
\end{align*}
for all $U \in \kD_{\tau,L+1}$, $h \in [0,\eta r/(c_Fm)]$, and $m \geq
m_*$.  The first and second error decrease with $m$ whereas the third
error increases with $m$.  We now demand that the two exponents on the
right coincide.  Under the ansatz $m = \zeta \, h^{-\alpha}$ for some
$\zeta$ and $\alpha \in (0,1)$, we obtain
\begin{equation*}
  \norm{\Psi^h (U) - \tilde \Phi_m^h ( U)}{\kY_1 }
  \leq (1+ c_\Psi) \, h^{\alpha \ell} \, \zeta^{-\ell} \, 
         \e^{-\tau\zeta^{1/q} h^{-\alpha/q}}
       + c_1  \, c_F \, \zeta \, h^{1-\alpha} \, 
         \e^{-\chi \zeta^{-1} h^{\alpha-1}}
\end{equation*}
with $\chi = c_2 r/c_F = c_2 \delta/(4 c_F)$.  Then the exponents
coincide provided $\tau \zeta^{1/q} h^{-\alpha/q} = \chi
\zeta^{-1}h^{\alpha-1}$, i.e., when
\begin{equation}
  \label{e.chooseAlpha}
  \alpha = \frac{q}{1+q} 
  \qquad \mbox{and} \qquad 
  \zeta = \Bigl( \frac\chi\tau \Bigr)^\alpha \,.
\end{equation}
This implies that $m(h)$ is given by \eqref{e.m}, and
\begin{equation}
  \norm{\Psi^h(U) - \tilde \Phi^h_{m(h)} (U)}{\kY_1}
  \leq \tilde c\, h^\nu \, \e^{-c_*  h^{-\frac{1}{1+q}}} 
  \label{e.hilbert-embedding-1}
\end{equation}
with $c_* = \tau \zeta^{1/q} = \tau^{\frac{q}{1+q}} \,
\chi^{\frac{1}{q+1}}$ and $\nu = {\min \{1, q\ell \}}/{(1+q)}$, and
where we possibly need to shrink $h_*>0$ to satisfy $m(h_*) \geq m_*$
and $h_* \leq \eta r/(c_F m(h_*)) = \eta/c_2 \, \chi^{1-\alpha} \,
(\tau h_*)^\alpha$.  Solving the latter inequality for $h_*$ leads to
the restriction that
\[
   h_* \leq \chi \, \tau^q \left( \frac{\eta}{c_2} \right)^{q+1} \,.
\]
A similar computation yields the form of $n(h)$ stated in \eqref{e.m}.
The exponential estimate \eqref{e.hilbert-embedding} is then obtained
by defining the modified vector field by \eqref{e.modVF} with
corresponding modified flow $\tilde \Phi^t(U)\equiv
\tilde\Phi^t_{m(h)} (U)$ and setting $c_\Phi = \tilde c \, h_*^\nu$.
\end{proof}

\begin{remark}
The formal expansions of both, the numerical method and the modified
vector field, contain powers of the unbounded operator $A$.
Therefore, the modified vector field cannot be written as a semilinear
Hamiltonian evolution equation of the form \eqref{e.HamPDE}.  If we
were simply interested in constructing the modified vector field, we
could avoid using spatial Galerkin truncation by setting up different
spaces for domain and range such that the modified vector field,
computed up to a given order, is continuous.  In fact, we need such
techniques in Section~\ref{ssec:formalBEA} to obtain the order
estimate \eqref{e.HtildeHClose} for the modified Hamiltonian which is
yet to be constructed.  In such a setting, however, we do not have a
theory of local existence of solutions for the modified differential
equation, so that we cannot obtain an approximate embedding of the
numerical method into a flow.
\end{remark}

\begin{remark} \label{r.WhyWorkInY1} 
The reason for constructing the modified vector field on a subspace of
$\kY_1$ rather than $\kY$ is that on general domains we can only
maintain a valid domain of definition of the nonlinearity $B(U)$ under
Galerkin truncation uniformly in $m\geq m_*$, and, in particular,
assert estimate \eqref{e.fm-estimate} by dropping down at least one
rung on the scale of spaces.  Similarly, we require data in
$\kD_{\tau,L+1}$ rather than $\kD_{\tau,L}$ for the exponential
estimates \eqref{e.hilbert-embedding} and \eqref{e.onestepHamExpEst}
because we want to define $\Psi^h$ and $\Psi^h_m$ uniformly in $h$ and
$m$ on general open sets of Gevrey spaces, not just on open balls.
This can only be done when constructing them as maps from
$\kD_{\tau,L+1}$ to $\kY_{\tau,L+1} \cap \kD_{\tau,L}^\delta$, see
\cite{OliverWulffRK,OliverWulffGalerkin}.
\end{remark}

Next, we show that in the Hamiltonian case the above construction also
yields a modified Hamiltonian which is approximately conserved under
the numerical time-$h$ map of the full semilinear evolution equation
\eqref{e.HamPDE}.

\begin{lemma}[Modified Hamiltonian for Gevrey class data]  
\label{l.gevrey-hamiltonian}
Under the conditions and in the notation of
Lemma~\ref{l.gevrey-embedding}, suppose further that \textup{(H0--4)}
and \textup{(RK4)} hold true.  Then, for sufficiently small $h_*>0$,
there exists a modified Hamiltonian $\tilde{H} \colon \kD_1^{\delta/2}
\times [0, h_*] \to\R$, defined up to a constant of integration, which
is analytic in $U\in \kD_1^{\delta/2}$ for every $h \in [0,h_*]$ and
such that the modified vector field from
Lemma~\ref{l.gevrey-embedding} satisfies $\tilde F = \JJ \nabla \tilde
H$.  Moreover, there exist constants $c_* \in
(0,\tau^{\frac{q}{q+1}}\chi^{\frac{1}{q+1}})$, with $\chi$ as in
Lemma~\ref{l.gevrey-embedding} and $c_{\tilde{H}}>0$ such that for
every $U \in \kD_{\tau,L+1}$ and $h \in [0,h_*]$,
\begin{gather}
  \lvert \tilde{H}(\Psi^h (U),h) - \tilde{H}(U,h) \rvert
  \leq c_{\tilde{H}} \, \e^{-c_* h^{-\frac{1}{1+q}}} \,.
  \label{e.approx-conservation}
\end{gather}
\end{lemma} 
 
\begin{proof}
By assumption (H1), the operator $\JJ^{-1}A$ is self-adjoint on $\kY$.
Since, by Lemma \ref{Le:JPm}, $\JJ^{-1}$ and $\P_m$ commute ,
$\JJ^{-1}A$ is also self-adjoint on $\P_m \kY$ with respect to the
restriction of the $\kY$-inner product to $\P_m \kY$.  Hence, the
linear part $\dot u_m = A_m u_m$ of \eqref{e.slwe-m} is Hamiltonian on
$\P_m \kY$.  Moreover, by (H2), the operator $\JJ^{-1}\D B(U)$ is
self-adjoint for each $U \in \kD^\delta$.  Hence, $\JJ^{-1}\D
B_m(u_m)$ is self-adjoint for each $u_m \in \kD^\delta \cap \P_m \kY$
so that, altogether, the vector field $f_m$ from \eqref{e.slwe-m} is
Hamiltonian as map from $ \kD^\delta \cap \P_m \kY$ to $\P_m \kY$ with
respect to the restriction of the $\kY$-inner product to $\P_m \kY$.
Since on each Galerkin subspace $\P_m \kY$ the numerical method
$\psi^h_m$ is symplectic, the Taylor coefficients $f_m^j$ of the
modified vector field $f_m^j$ are also Hamiltonian; see, e.g.,
\cite{BeGi1994, HaLu2002,LeRe04}.  Moreover, the operator $\JJ^{-1} \D
f_m^j(u_m)$ is self-adjoint with respect to the restriction of the
$\kY$-inner product to $\P_m \kY$ for each $u_m \in \kD^\delta \cap
\P_m \kY$ and the same holds true for $\JJ^{-1} \D \tilde
f^{n(h)}_{m(h)}(u_m)$.

As we argued in the proof of Lemma~\ref{l.gevrey-embedding}, $\P_mU
\in \kD^\delta$ for $U \in \kD_1^{\delta/2}$, so that $\JJ^{-1} \D
\tilde F(U)$ is self-adjoint with respect to the $\kY$-inner product
for each $U \in \kD_1^{\delta/2}$.  By assumption (H4), the set
$\kD^{\delta/2}_1$ is simply connected and star-shaped.

Therefore, we can proceed as in the proof of Lemma~\ref{l.H2}.  We fix
$U^0 \in \kD_1$ such that $\kD_1^{\delta/2}$ is star shaped with
respect to $U^0$ and define
\[
  \tilde H(U) 
  = \int_0^1 \langle \JJ^{-1} \tilde F (t\, U + (1-t) \, U^0), 
    U-U^0 \rangle \, \d t \,.
\]
This modified Hamiltonian $\tilde H$ is well-defined and analytic on
$\kD^{\delta/2}_1$.  Moreover, the steps taken in \eqref{e.HamComput}
still apply so that $\tilde H$ is invariant under the modified flow
with
\begin{equation*}
  \tilde F = \JJ \nabla \tilde H \,.
\end{equation*}

To prove \eqref{e.approx-conservation}, we decrease $h_*$ such that
the right hand side of \eqref{e.hilbert-embedding} is smaller than
$\delta/4$ for every $U \in \kD_{\tau,L+1}$.  This is to ensure that
$\Psi^h(U) \in \kD_1^{\delta/2}$, so that $\tilde H$ is defined at
$\Psi^h(U)$.  Then, using the mean value theorem, the bound on the
modified vector field given by \eqref{e.modVF},
\eqref{e.hilbert-embedding}, and the invertibility of $\JJ^{-1}$, we
estimate, for $h \in [0,h_*]$ and $U \in \kD_{\tau,L+1}$, that
\begin{align*}
  \bigl\lvert 
    \tilde H (\Psi^h(U)) - \tilde H (\tilde \Phi^h (U))
  \bigr\rvert
  & \leq \norm{\nabla \tilde{H}}{\kCb(\kD_1^{\delta/2}; \kY_1)} \,
	 \norm{\Psi^h(U) - \tilde \Phi^h(U)}{\kY_1} 
         \notag \\
  & \leq \norm{\JJ^{-1} \tilde F }{\kCb(\kD_1^{\delta/2}; \kY_1)} \, 
         \norm{\Psi^h(U) - \tilde \Phi^h(U)}{\kY_1 }
         \notag \\
  & \leq \norm{\JJ^{-1}}{\kE(\kY_1)} \,
         \norm{\tilde F}{\kCb(\kD_1^{\delta/2}; \kY_1)} \, 
         c_{\tilde\Phi} \, \e^{-c_* h^{-1/(1+q)}} 
         \notag \\
  & \leq \norm{\JJ^{-1}}{\kE(\kY_1)} \,
         c_{\tilde F} \, m(h) \, c_{\tilde\Phi} \, 
         \e^{-c_* h^{-1/(1+q)}} \,.
\end{align*}
Since $\tilde H$ is conserved under the modified flow, choosing $m$ as
in \eqref{e.m}, we obtain
\[
  \lvert \tilde{H}(\Psi^h (U)) - \tilde{H}(U) \rvert
  \leq \tilde{c} \, h^{- \tfrac{ q}{1+q}} \, 
       \e^{-c h^{-\frac{1}{1+q}}} \,.
\]
Dominating the algebraic prefactor by fractional exponential decay,
this inequality implies \eqref{e.approx-conservation} with a possibly
decreased value for $c_*>0$ as compared to its value in
Lemma~\ref{l.gevrey-embedding}.
\end{proof}

What is still missing is the proof of the $O(h^p)$-closeness of the
modified to the original Hamiltonian.  In a first attempt to prove
such a result, we write
\begin{equation}\label{e.h-closeness}
  \lvert H (U) - \tilde H  (U) \rvert
  \leq \lvert H (U) - H (\P_m U) \rvert +
         \lvert H (\P_m U) - \tilde H (U) \rvert \,,
\end{equation}
where $m=m(h)$ is as in \eqref{e.m}.  Under the assumptions of
Lemma~\ref{l.gevrey-hamiltonian}, the first term on the right is
exponentially small for $U \in \kD_{\tau,L+1}$.
 
To estimate the second term of \eqref{e.h-closeness}, choose some
fixed $U^0 \in \kD_1$ such that $\kD_1^{\delta/2}$ is star-shaped with
respect to $U^0$, and set $H(U^0) = \tilde H(U^0)$.  The naive choice
is then to employ \eqref{e.tildefa-tildeF} with $a=p$ so that $\tilde
F^a_m = \tilde f^a_{m} \circ \P_{m} = F \circ \P_{m}$.  Integrating
$\tilde F-F\circ \P_m$, we obtain the estimate
\begin{equation}
  \lvert H(u_m) - \tilde H_m(u_m) \rvert
  \leq O(m^{p+1}h^p) = O(h^{\frac{p-q}{q+1}}) 
  \label{e.suboptimal}
\end{equation}
for every $u_m = \P_m U$ and $U \in \kD_{1}^{\delta/2}$.  This
estimate is weaker than the expected $O(h^p)$.
 
A closer inspection reveals that, in the context of the semilinear
evolution equation \eqref{e.HamPDE}, the second inequality of
\eqref{e.fh-f} in the proof of Lemma~\ref{l.fh-f} is too weak: when
estimating the Taylor coefficients $f^j_m$ of the modified vector
field in some fixed Hilbert space norm, the unboundedness of the
operators $A$ contained therein will introduce a factor $m^j$, see
\eqref{e.fj}.  This propagates into the proof of \eqref{e.suboptimal}.
Note, however, that these estimates are simply about consistency, not
about constructing a flow.  Thus, we can afford to lose smoothness
rather than order.  In other words, we can estimate the Taylor
coefficients of the modified vector field as maps from one Hilbert
space into another with a weaker norm.  This will be detailed in the
next section.


\subsection{Modified vector fields on Hilbert spaces}
\label{ssec:formalBEA}

In this section, we present a more subtle estimate on the difference
between the original Hamiltonian and the modified Hamiltonian.  The
main difference to the derivation of \eqref{e.suboptimal} in the
previous section is that we consider the expansion coefficients of the
numerical method and of the modified vector field as maps between
different rungs on our scale of Hilbert spaces such that the loss of
smoothness is carefully accounted for.

We begin by establishing the necessary functional setting for the
analysis of modified vector fields on Hilbert spaces.  We then review
a result from \cite{OliverWulffGalerkin} on the Galerkin projection
error for the numerical time-$h$ maps.  This estimate is then
propagated into an estimate on the difference between the full and the
modified vector field, which finally implies a corresponding estimate
on the difference between the exact Hamiltonian and the modified
Hamiltonian.

In this section, we work directly with the standard construction of
the modified vector field.  Namely, for $\ell = 1, \dots, K+1$, we
write
\begin{equation}
  \label{e.Gl}
  G^\ell = \frac{\partial_h^\ell \Psi^h}{\ell!} \bigg|_{h=0}
\end{equation}
to denote the $\ell$th coefficient of the expansion of $\Psi^h$ in
powers of $h$, and define the analog of the expansion coefficients
\eqref{eq:jModf} for the modified vector field on Hilbert spaces as
follows: we set $F^1 \equiv G^1$ and define
\begin{equation}
  F^\ell 
  = G^\ell - \sum_{i=2}^\ell \frac1{i!} 
    \sum_{\ell_1+\ldots + \ell_i = \ell}
    \D_{\ell_1} \cdots \D_{\ell_{i-1}} F^{\ell_i} \,,
  \label{e.mod-vect}
\end{equation}
for $\ell = 2, \dots, K+1$, where the sum ranges over indices $\ell_i
\geq 1$ for all $i$, where $\D_j G = \D G \, F^j$.  We also recall
from Section~\ref{ssec:embedFlow} that $g_m^\ell$ and $f_m^\ell$
denote the $\ell$th coefficients of the expansions of the projected
numerical method and the corresponding modified vector field,
respectively, and set $G_m^\ell \equiv g_m^\ell \circ \P_m$ and
$F_m^\ell \equiv f_m^\ell \circ \P_m$.

In the notation of condition (B1), we abbreviate $\kU_\kappa =
\kD_{\kappa}$ for $\kappa = 1, \dots, K+1$.  Then the regularity
results Theorem~\ref{t.h-diff-unif} on $\Psi^h$ and
Theorem~\ref{t.wm_psim-uniform} on $\Psi_m^h$ imply that, in
particular, there exists $m_*$ such that for all $m \geq m_*$ and
$\ell = 1, \dots, K+1$,
\begin{equation}
  \label{e.class-Gl}
  G^\ell, G^\ell_m\in
  \bigcap_{\substack{j+k \leq N \\ \ell \leq k \leq K+1}}
  \kCb^j (\kU_k; \kY_{k-\ell}) \,.
\end{equation}
Moreover, bounds in the norms associated with \eqref{e.class-Gl} are
uniform in $m\geq m_*$.

In the following, we will state such bounds on vector fields in terms
of the three-parameter family of norms
\begin{equation}
  \label{e.snormg}
  \snorm{g}{N,K,S} 
  = \max_{\substack{j+k \leq N \\ S \leq k \leq K}}
    \norm{\D^j g}{\kC(\kU_k; \kE^j(\kY_k,\kY_{k-S}))} 
\end{equation}
for $1 \leq S \leq K \leq N$.  The parameter $S$ plays the role of a
loss of smoothness index as it forces the image of the map be
estimated at least $S$ rungs down the scale.  We can then prove a
simple result on the regularity of the expansion coefficients of the
modified vector field.

\begin{lemma}[Modified vector field on a scale of Hilbert spaces]
\label{l.mod-vect}
Assume that $G^1, \dots, G^{K+1}$ are of class \eqref{e.class-Gl}.
Then the vector fields $F^1, \dots, F^{K+1}$ defined by
\eqref{e.mod-vect} are also of class \eqref{e.class-Gl}.
\end{lemma}

\begin{proof}
The proof is based on the simple fact that $F \in
\kCb^n(\kY_{i+j},\kY_{j})$ and $G \in \kCb^{n+1}(\kY_{j},\kY)$ implies
that $\D G\, F\in\kCb^n(\kY_{i+j},\kY)$.  Thus, it remains to observe
that repeated application to the terms in the inner sum of
\eqref{e.mod-vect} causes all loss indices to always sum up to $\ell$.
We proceed by induction in $\ell$.  The case $\ell=1$ does not require
proof.  Assume therefore that $\ell>1$ and the lemma is proved up to
index $\ell-1$.  For $\ell> \ell_1\geq 1$, we estimate, with $G \equiv
\D_{\ell_2} \ldots \D_{\ell_{i-1}} F^{\ell_i}$ that
\begin{align}
 & \snorm{\D_{\ell_1} G}{N,K+1,\ell}
   = \snorm{\D G \, F^{\ell_1}}{N,K+1,\ell}
      \notag \\
  & = \max_{\substack{j + k \leq N \\ \ell \leq k \leq K+1}}
      \norm{\D^j(\D G \, F^{\ell_1})}%
           {\kC(\kU_k; \kE^j(\kY_k,\kY_{k-\ell}))}
      \notag \\
  & \leq 2^N \max_{\substack{j+k \leq N \\ \ell \leq k \leq K+1}}
         \norm{\D^{j+1} G}%
              {\kC(\kU_{k-\ell_1}; \kE^{j+1}(\kY_{k-\ell_1},
               \kY_{k-\ell}))} 
         \max_{\substack{j+k \leq N \\ \ell_1 \leq k \leq K+1}}
         \norm{\D^{j} F^{\ell_1}}%
              {\kC(\kU_k; \kE^j(\kY_k, \kY_{k-\ell_1}))}
      \notag \\
  & = 2^N \max_{\substack{j+k \leq N+1-\ell_1 \\ 
                      \ell-\ell_1 \leq k \leq K+1-\ell_1}}
         \norm{\D^{j} G}%
              {\kC(\kU_k; \kE^j(\kY_k, \kY_{k-(\ell-\ell_1)}))} \,
      \snorm{F^{\ell_1}}{N,K+1,\ell_1} 
      \notag \\
  & \leq 2^N \, \snorm{G}{N,K+1-\ell_1,\ell-\ell_1} \,
           \snorm{F^{\ell_1}}{N,K+1,\ell_1}  \,,
  \label{e.prod-est-first-step}
\end{align}
provided that $G$ is of class \eqref{e.class-Gl} with $K$ replaced by
$K-\ell_1$ and $\ell$ replaced by $\ell-\ell_1$.  Here the first
inequality is based on the product rule and selective weakening of the
norm on the domain spaces, thereby increasing the respective operator
norms.  The identity between the third and the fourth line is achieved
by redefining $j+1$ as $j$ and $k-\ell_1$ as $k$.  The final
inequality holds because $\ell_1 \geq 1$, so that we are strictly
extending the range of the running indices.  We note that the second
term in the final line of \eqref{e.prod-est-first-step} is bounded by
the induction hypothesis.  The estimation of the first term in the
final line of \eqref{e.prod-est-first-step} can now be done
recursively to resolve the entire product $\D_{\ell_1} \cdots
\D_{\ell_{i-1}} F^{\ell_i}$ from the inner sum of \eqref{e.mod-vect}
in terms of quantities which are bounded by the induction hypothesis.
This is always possible, because at the $k$th step of this process we
lose $\ell_k$ rungs of smoothness, and the sum of the loss indices
satisfies $\ell_1+ \dots + \ell_i = \ell$ by construction.
\end{proof}

We now aim to derive estimates on the difference between $F^\ell$ and
$F_m^\ell$ with respect to the same type of norm.  In
\cite{OliverWulffGalerkin}, we already obtained a related result on
the difference between $\Psi^h$ and $\Psi_m^h$ which we can start
from.  Setting $\kI=(0,h_*)$, $\kU = \kD_{K+1}$, and $\kX =
\kY_{K+1}$, we define the norm
\begin{equation}
  \norm{\Psi}{N,K} 
  = \max_{\substack{j+k \leq N \\ 
           \ell \leq k \leq K}}
    \norm{\D_U^j \partial_h^\ell \Psi}{\kL_\infty
          (\kU \times \kI; \kE^j(\kX; \kY_{k-\ell}))} 
  \label{e.NK-norm}
\end{equation}
for $0 \leq K \leq N$.  Then the stability of the numerical method on
a scale of Hilbert spaces under spectral truncation can be formulated
as follows.

\begin{lemma}[{\cite[Theorem 3.7]{OliverWulffGalerkin}}]
\label{l.projection-error-scale}
Assume \textup{(A)}, \textup{(B1)}, and \textup{(RK1--3)}.  Then there
is $h_*>0$ such that for every $0 \leq S \leq K+1$,
\begin{equation}
  \norm{\Psi^h - \Psi^h_m}{N-1-S,K+1-S} = O(m^{-S}) 
  \label{e.psi-m-s-unif}
\end{equation}
as $m \to \infty$.  The order constants depend only on the bounds
afforded by \textup{(B1)}, \eqref{e.Rk}, \eqref{e.RKShA}, on the
coefficients of the method, and on $\delta$.
\end{lemma}
 
We note that Lemma~\ref{l.projection-error} above already provided us
with an exponential estimate on $\Psi^h - \Psi^h_m$ for Gevrey regular
data, whereas Lemma~\ref{l.projection-error-scale} here also asserts
bounds on derivatives with respect to $h$ and $U$; the proof is
correspondingly more complicated even in spaces of finite order of
smoothness and can be found in \cite{OliverWulffGalerkin}.

To proceed, we observe that Lemma~\ref{l.projection-error-scale} holds
with $K+1$ in the statement of the lemma replaced by any $\kappa$
between $S+1$ and $K+1$, with $K$ as defined in condition (B1).  Note
that in the definition of the norm \eqref{e.NK-norm}, we must
correspondingly read $\kX$ and $\kU$ as $\kY_\kappa$ and $\kU_\kappa$,
respectively.  Then, specializing to the particular value $k=\kappa-S$
in the definition of the norm appearing in \eqref{e.psi-m-s-unif}, we
obtain
\[
  \max_{\substack{j+\kappa \leq N-1 \\ S+\ell \leq \kappa}}
  \norm{\D_U^j \partial_h^\ell (\Psi^h - \Psi^h_m)}%
    {\kL_\infty(\kU_\kappa \times \kI;
     \kE^j(\kY_\kappa;\kY_{\kappa-S-\ell}))} 
  = O(m^{-S}) \,.
\]
Thus, fixing $\ell \in 1, \dots, K+1-S$ and taking the maximum over
the allowed range $\kappa = S+\ell, \dots, K+1$, we can write
\begin{equation}
  \max_{\substack{j+\kappa \leq N-1 \\ S+\ell \leq \kappa \leq K+1}}
  \norm{\D_U^j \partial_h^\ell (\Psi^h - \Psi^h_m)}%
       {\kL_\infty(\kU_\kappa \times \kI;
        \kE^j(\kY_\kappa;\kY_{\kappa-S-\ell}))} 
  = O(m^{-S}) \,.
  \label{e.psi-m-s2}
\end{equation}
Due to the definition of $G^\ell$ in \eqref{e.Gl}, this directly
implies that
\begin{equation}
  \snorm{G^\ell - G_m^\ell}{N-1,K+1,S+\ell} = O(m^{-S}) \,.
  \label{e.Gl-diff}
\end{equation}
 
\begin{lemma}[Stability of the modified vector field]
\label{l.mod-vect-stab}
Suppose $G^1, \dots, G^{K+1}$ and $\bar G^1, \dots, \bar G^{K+1}$ are
of class \eqref{e.class-Gl}, and let $F^\ell$ and $\bar F^\ell$ denote
expansion coefficients of the respective associated modified vector
fields defined via \eqref{e.mod-vect}.  Then, for $S \in 1, \dots, K$
and every $\ell \in 1, \dots, K+1-S$,
\[
  \snorm{F^\ell - \bar F^\ell}{N-1,K+1,S+\ell} 
  \leq c \, \max_{1 \leq k \leq \ell} \, 
       \snorm{G^k - \bar G^k}{N-1,K+1,S+k} \,,
\]
where the constant $c$ depends on the bounds on $G^k$ and $\bar G^k$
in the norm $\lvert \, \cdot \, \rvert_{N-1,K+1,k}$ for $k = 1, \dots,
\ell$.
\end{lemma}

\begin{proof}
The proof follows the same route as the proof of
Lemma~\ref{l.mod-vect}.  We set $\bar\D_{\ell} G \equiv \D G \bar
F_{\ell}$.  The crucial estimate corresponding to
\eqref{e.prod-est-first-step} then takes the form
\begin{align*}
  \snorm{\D_{\ell_1} G - \bar \D_{\ell_1} \bar G}{N-1,K+1,S+\ell}
  & \leq \snorm{\D G \, (F^{\ell_1} - \bar F^{\ell_1})}{N-1,K+1,S+\ell} 
    \notag \\  
  & \quad
       + \snorm{\D (G - \bar G) \, \bar F^{\ell_1}}{N-1,K+1,S+\ell}
    \notag \\
  & \leq 2^N \, \snorm{G}{N-1,K+1-S-\ell_1,\ell-\ell_1} \,
         \snorm{F^{\ell_1} - \bar F^{\ell_1}}{N-1,K+1,S+\ell_1}
    \notag \\
  & \quad
       + 2^N \, \snorm{G - \bar G}{N-1,K+1-\ell_1,S+\ell-\ell_1} \, 
         \snorm{\bar F^{\ell_1}}{N-1,K+1,\ell_1} \,,
\end{align*}
where the estimates in the last inequality follow from
\eqref{e.prod-est-first-step}.  This reasoning can again be applied
iteratively to resolve the entire difference of products of the form
$\D_{\ell_1} \cdots \D_{\ell_{i-1}} F^{\ell_i}$, where we note that
the loss indices now add up to exactly $S+\ell$ as required.
\end{proof}

We now turn our attention to the optimally truncated modified vector
field $\tilde F \equiv \tilde f_m \circ \P_m$ where $m=m(h)$ is as in
\eqref{e.m}.  This is the same modified vector field which gives rise
to the modified Hamiltonian in Lemma~\ref{l.gevrey-hamiltonian}.  Then
the difference between $\tilde F$ and $F$, the exact vector field of
the semilinear evolution equation \eqref{e.pde}, can be estimated as
follows.

\begin{lemma} \label{prop:modVFpcloseF} 
Suppose conditions \textup{(A)}, \textup{(B0--2)}, and
\textup{(RK1--3)} are satisfied with $K+1 \geq P$, where
\begin{equation}
  P = \lceil p(q+1)^2/q + q \rceil + 1 \,,
  \label{e.pcondition}
\end{equation}
$q$ is defined in \textup{(B2)}, and $p$ is the order of the numerical
method.  Then the difference between the original vector field $F$ and
the modified vector field $\tilde{F}$ from
Lemma~\ref{l.gevrey-embedding} satisfies
\begin{equation} 
  \label{e.tildefFclose}
  \norm{\tilde{F}-F}{\kCb(\kD_{P};\kY)} = O(h^p) \,.
\end{equation}
\end{lemma}

\begin{proof}
We define two intermediate vector fields.  Let $\tilde F^a$ denote the
modified vector field of the numerical method applied to the original
semilinear evolution equation \eqref{e.pde} computed up to order
$a\leq K+1$.  Formally, as in the finite dimensional case, it has an
expansion of the form \eqref{eq:tildeFSeries},
\begin{equation}
  \tilde{F}^a = F + \sum_{j=p}^{a-1} h^j \, F^{j+1} \,.
  \label{eq:tildeFSeries.2}
\end{equation}
Note that $F = \tilde{F}^p$.  Due to \eqref{e.class-Gl} and
Lemma~\ref{l.mod-vect}, all coefficients $F^k$ in the expansion
\eqref{eq:tildeFSeries.2} and consequently $\tilde{F}^a$ are also of
class \eqref{e.class-Gl} with $\ell=a$.

We now choose $h_*$ and $m_*$ as in the proof of
Lemma~\ref{l.gevrey-embedding} and suppose $m \geq m_*$, so that the
projected modified vector fields are well defined.  Let $\tilde F_m^a
\equiv \tilde f_m^a \circ \P_m$ denote the corresponding modified
vector field of the projected system \eqref{e.slwe-m}, again up to
order $a$, with $m = O(h^{-q/(q+1)})$ as in \eqref{e.m}.  By
Lemma~\ref{l.mod-vect} applied to $G^\ell_m$, the vector field $\tilde
F_m^a$ is also of class \eqref{e.class-Gl}.

We now decompose
\begin{equation}
  \label{e.decompose}
  F - \tilde F = (F - \tilde F^a) 
  + (\tilde F^a - \tilde F_m^a) + (\tilde F_m^a - \tilde F) \,.
\end{equation}
We show that when $a$ is chosen as
 \begin{equation}
 \label{e.a}
  a = \lceil p (q+1) + q \rceil \,,
\end{equation}
each term on the right is $O(h^p)$ in appropriate norms.
 
A bound on the first difference on the right of \eqref{e.decompose}
follows directly from the definition of $\tilde F^a$ in
\eqref{eq:tildeFSeries.2}.  Using the norms defined in
\eqref{e.snormg},
\begin{equation*}
  \snorm{F - \tilde F^a}{N-1,K+1,a}
  \leq h^p \sum_{j=p}^{a-1} h^{j-p} \, \snorm{F^{j+1}}{N-1,K+1,a} \,.
\end{equation*}
By Lemma~\ref{l.mod-vect}, all norms in the right hand sum are finite,
so that, for $a \in 1, \dots, K+1$,
\begin{equation}
  \label{e.FFtildeClose}
  \snorm{F - \tilde{F}^a}{N-1,K+1,a} = O(h^p) \,.
\end{equation}

To estimate the second difference on the right of \eqref{e.decompose},
we apply Lemma~\ref{l.mod-vect-stab} with $\bar G^\ell = G_m^\ell$,
which we recall are also of class \eqref{e.class-Gl}, so that
\eqref{e.Gl-diff} applies, yielding
\begin{equation}
  \label{e.split2}
  \snorm{\tilde F^a - \tilde F_m^a}{N-1,K+1,a+S} = O(m^{-S}) 
\end{equation}
for $S \in 1, \dots, K+1-a$.
 
A bound on the third difference on the right of \eqref{e.decompose} is
provided by \eqref{e.tildefa-tildeF}, namely
\begin{equation}
  \label{e.split3}
  \norm{\tilde F_{m(h)}^a - \tilde F}{\kC(\kD_1^{\delta/2};\kY_1)}
  = O(h^a \, {m(h)}^{a+1}) \,.
\end{equation}

We now seek conditions under which the estimates \eqref{e.split2} and
\eqref{e.split3} are of order $h^p$.  Due to \eqref{e.m}, $m = {m(h)}
= O(h^{-q/(1+q)})$, so that the requirement $m^{-S}=O(h^p)$ leads to
the choice
\[
  S = \biggl\lceil p \, \frac{1+q}q \biggr\rceil \,.
\]
Similarly, the requirement $O(h^a \, m^{a+1}) = O(h^p)$ is equivalent
to
\[
  O(h^a m^{a+1}) = O(h^{a - \frac{q(a+1)}{q+1}})
  = O(h^{\frac{a}{q+1} - \frac{q}{q+1}})
  = O(h^p) \,,
\]
which leads to the choice \eqref{e.a}.

Since $P$ is defined such that $P \geq S+a$, see \eqref{e.pcondition},
the above estimates for the three terms of the decomposition
\eqref{e.decompose} imply \eqref{e.tildefFclose}.
\end{proof}

\begin{remark}
If we can change $\tilde{F}$ so that its leading order linear part is
$A$ rather than $A \P_{m(h)}$, then Theorem~\ref{t.GenMain} still
applies.  This is true since, by \eqref{e.qm-estimate}, for $U^0 \in
\kY_{\tau,L+1}$ the differences between the two modified vector fields
and their flows up to time $h$ are exponentially small in the
$\kY_1$-norm.
\end{remark}

In the Hamiltonian case, the previous result on $O(h^p)$-closeness of
true and modified vector field carries over to a statement on
$O(h^p)$-closeness of the corresponding Hamiltonians.

\begin{lemma} \label{l.H-TildeH} 
Under the assumptions of Lemma~\ref{prop:modVFpcloseF} suppose that,
in addition, the semilinear evolution equation is Hamiltonian
satisfying \textup{(H0--4)} and that the numerical method is
symplectic, i.e., satisfies \textup{(RK4)}, and is of order $p$.  Then
the modified Hamiltonian $\tilde{H}$ from
Lemma~\ref{l.gevrey-hamiltonian} can be chosen such that
\[
  \norm{H - \tilde{H}}{\kCb(\kD_{P};\R)} = O(h^p) \,.
\]
\end{lemma}

\begin{proof}
By (H4), $\kD_P$ is star-shaped with respect to some $U^0 \in \kD_P$.
Since the modified Hamiltonian from Lemma~\ref{l.gevrey-hamiltonian}
is defined only up to a constant, we can choose the constant of
integration such that $H(U_0) = \tilde H(U_0)$.  Then, following the
proof of Lemma~\ref{l.H2}, we estimate, for any $U \in \kD_{P}$,
\begin{align*}
  \lvert \tilde{H}(U) - H(U) \rvert
  & \leq \biggl| 
           \int_0^1 \langle 
                      \JJ^{-1} (F-\tilde{F})(t \, U + (1-t) \, U_0), 
                      U-U_0
                    \rangle \, \d t
         \biggr| \\
  & \leq \norm{\JJ^{-1}}{\kE(\kY)} \,
         \norm{F - \tilde{F}}{\kCb(\kD_{P}; \kY)} \,
         \sup_{U \in \kD_{P}} \norm{U-U^0}{\kY}  \,.
\end{align*}
Since $\kD_{P}$ is $\kY$-bounded, Lemma~\ref{prop:modVFpcloseF}
implies that the right hand side is $O(h^p)$.
\end{proof}
 

\section{Lower estimates in an example: A nonlinear Schr\"odinger equation}
\label{s.LowerEst}

In this section we set up a counter-example, motivated by
\cite{MaSch2001}, which shows that analyticity of the initial data is
necessary to achieve an embedding of the implicit midpoint rule into a
Hamiltonian flow.

\subsection{Model evolution equation}

We work with functions $u \colon [0,\infty) \to \kl_2(\N_0; \C)$ whose
components may be interpreted as the Fourier coefficients of a
function defined on the circle.  Further, we write $u=v + \i w$ to
identify the real and imaginary parts of the components of $u$.
 
We now define the Hamiltonian 
\begin{equation*}
  H = \frac12 \sum_{k=1}^\infty \omega_k \, \lvert u_k \rvert^2 
      +w_0 \Re f(u)
\end{equation*}
with  $\omega_k\geq 0$ and
\begin{equation*}
  f(u) = f(u_1, u_2, \dots) = \sum_\alpha c_\alpha \, u^\alpha \,,
\end{equation*}
where $\alpha=(\alpha_1, \alpha_2, \ldots) \in \N_0^{\N}$ is a
multi-index, each $c_\alpha$ is a real coefficient, $u^\alpha =
u_1^{\alpha_1} \, u_2^{\alpha_2} \, \cdots$ as usual, and the
summation is over all multi-indices with $\alpha_0 = 0$ and a finite
number of non-zero coefficients $\alpha_k$ for $k \geq 1$.

In terms of $v$ and $w$, this defines a Hamiltonian system
\begin{equation}
  \frac\d{\d t}
  \begin{pmatrix}
    v \\ w
  \end{pmatrix}
  = \JJ \, \nabla H(v,w) \,,
  \label{e.ham-model}
\end{equation}
where $\JJ$ is  the standard symplectic structure matrix
\[
  \JJ = 
  \begin{pmatrix}
    0  & \id \\ -\id & 0
  \end{pmatrix} \,.
\]
For $k=0$, we obtain by direct computation that
\begin{subequations}
\label{e.v-w}
\begin{align}
  \dot v_0 & =  \frac{\partial H}{\partial w_0} = \Re f(u) \,, 
  \label{e.v0}
   \\
  \dot w_0  &= -\frac{\partial H}{\partial v_0} = 0 \,.
  \label{e.w0}
\end{align}
For $k \geq 1$, 
\begin{align}
  \dot v_k =  \frac{\partial H}{\partial w_k} 
           = \omega_k \, w_k 
             + w_0 \, \Re \frac{\partial f(u)}{\partial w_k} \,,
             \label{e.vk} \\
  \dot w_k = \frac{\partial H}{\partial v_k} 
           = -\omega_k \, v_k 
             - w_0 \, \Re \frac{\partial f(u)}{\partial v_k} \,.
             \label{e.wk}
\end{align}
\end{subequations}

In all of the following, we assume the initial condition $w_0(0)=0$ so
that, due to \eqref{e.w0}, $w_0(t)=0$ for all $t \geq 0$.  Then
\eqref{e.vk} and \eqref{e.wk} combine into
\begin{equation*}
  \dot u_k = \i \omega_k \, u_k
\end{equation*}
for $k \geq 1$, which is immediately solved as
\begin{equation*}
  u_k(t) = u_k(0) \, \e^{\i \omega_k t} \,.
  \label{e.uk-sol}
\end{equation*}
Substituting \eqref{e.uk-sol} back into \eqref{e.v0} and integrating
in time, we obtain
\begin{align*}
  v_0(t) 
  & = v_0(0) + \Re \sum_\alpha c_\alpha \, u^\alpha(0) 
      \int_0^t \e^{\i \sum_k \omega_k \alpha_k t} \, \d t
      \notag \\
  & = v_0(0) + \Re \sum_\alpha c_\alpha \, u^\alpha(0) \, 
      \frac{\e^{\i \sum_k \omega_k \alpha_k t} - 1}%
           {\i \sum_k\omega_k \, \alpha_k} \,.
\end{align*}


\subsection{Implicit midpoint time discretization}

We write $u^n = v^n + \i w^n$ to denote a time-discretization of
\eqref{e.v-w}.  Specifically, one step of the implicit midpoint time
discretization with step size $h$ applied to \eqref{e.v0} and
\eqref{e.w0} takes the form
\begin{subequations}
\begin{gather}
  v_0^1 = v_0^0 + h \, \Re f((u^1+u^0)/2) \,, 
  \label{e.V0} \\
  w_0^1 = w_0^0 \,.
\end{gather}
\end{subequations}
Thus, assuming that $w_0$ is zero initially, $w^n_0$ remains zero in
every step.  Then, for $k \geq 1$, we obtain
\begin{equation*}\label{e.uk-IMPR}
  u_k^1 = u_k^0 + h \, \i \omega_k \, \frac{u_k^1+u_k^0}2
          + h \, \frac{w_0^1+w_0^0}2 \, (\dots) 
\end{equation*}
so that, noting that the last term is zero, we can write
\begin{gather}
  u^1 = \mS(hA) \, u^0
  \label{e.Uk}
\end{gather}
where $\mS(hA) = \diag (\ms_1, \ms_2, \dots)$ with
\begin{equation}
  \ms_k = \ms_k(h) = \frac{1+\tfrac12 \i \omega_k h}%
                          {1-\tfrac12 \i \omega_k h} \,.
  \label{e.sk}
\end{equation}
Plugging \eqref{e.Uk} into \eqref{e.V0}, we can write
\begin{gather}\label{e.V_0^1}
  v_0^1 = v_0^0 + h \Re f(\mG u^0)
        = v_0^0 + h \Re \sum_\alpha c_\alpha \, (\mG u^0)^\alpha 
\end{gather}
where $\mG = \mG(h) =  \diag (\mg_1, \mg_2, \dots)$ with
\begin{equation}\label{e.gk}
  \mg_k = \mg_k(h) = \frac{1+\ms_k}2
      = \frac1{1-\tfrac12 \i \omega_k h} \,.
\end{equation}


\subsection{Modified vector field}

To obtain an expression for the modified vector field for initial data
$w_0(0)=0$, we make the ansatz
\begin{gather}
  \dot {\tilde v}_0 = \Re(\tilde f(\tilde u))
  \qquad \text{with} \qquad
  \tilde f(u) = \sum_{\alpha} \tilde c_{\alpha} \, u^\alpha
  \label{e.tildev0}
\end{gather}
and that, for $k \geq 1$,
\begin{gather}
  \dot {\tilde u}_k = \i \, \tilde\omega_k \, \tilde u_k \,.
  \label{e.tildeuk}
\end{gather}
Integrating \eqref{e.tildeuk} from $0$ to $h$, identifying $\tilde
u(0) \equiv u^0$, and equating the resulting expression with
\eqref{e.Uk}, we immediately obtain that
\begin{gather}
  \ms_k = \e^{\i \tilde\omega_k h} \,.
  \label{e.omegak}
\end{gather}
Similarly, inserting the solution of \eqref{e.tildeuk} into
\eqref{e.tildev0}, integrating from $0$ to $h$, and equating the
resulting expression with \eqref{e.V_0^1}, we find
\begin{gather}
  \label{e.coeff}
  h \Re \bigl[ c_\alpha \, (\mG u^0)^\alpha \bigr]
  = \Re \biggl[ \tilde c_{\alpha} \,
    \frac{\e^{\i \sum_k \alpha_k\tilde\omega_k h} - 1}%
         {\i \sum_k \alpha_k \, \tilde\omega_k} \,
    (u^0)^\alpha
    \biggr] \,.
\end{gather}
Therefore, to satisfy \eqref{e.coeff}, we have to find $\tilde
c_\alpha$ such that
\begin{gather}
  h \, c_\alpha \,\mg^\alpha 
  = \tilde c_\alpha \, 
    \frac{\e^{\i \sum_k \alpha_k \tilde\omega_k h} - 1}%
         {\i \sum_k \alpha_k \, \tilde\omega_k} \, .
  \label{e.necessary}
\end{gather}
Note that the left hand side of this equation only vanishes when $h=0$
or $c_\alpha=0$.


\subsection{Resonances}

Equation \eqref{e.necessary} can be solved for $\tilde c_\alpha$
unless the numerator on the right hand side is zero.  Let us call this
a resonance.  Due to \eqref{e.omegak}, we can write the condition for
resonance as $\ms^\alpha = 1$.  Let us look for particular resonances
between three consecutive ``wave numbers'' such that
\begin{gather}
  \label{e.3resonance}
  \ms_{k-1} \, \ms_k \, \ms_{k+1} = 1 \,.
\end{gather}
Due to the definition of $\ms_k$ in \eqref{e.sk}, this condition reads
\begin{equation}
  \label{e.3Res}
  4 \, (\omega_k + \omega_{k+1} + \omega_{k-1}) 
  = h^2 \, \omega_{k-1} \, \omega_k \, \omega_{k+1} \,.
\end{equation}

We consider the case $\omega_k=k^2$ so that \eqref{e.v-w} is a
nonlinear Schr\"odinger equation.  We identify $u \equiv (v,w) \in
\kh_1\times \kh_1=\kY$, where $\kh_\ell = \kh_\ell(\N_0; \R)$.  Thus,
$\kY$ is the Sobolev space $\kH_{\ell}(\S^1;\C)$ in Fourier
coordinates.  Further, \eqref{e.v-w} is of the form \eqref{e.pde} with
$A$ the Laplacian in Fourier coordinates, i.e., $(A u)_k = \i \omega_k
u_k$.  Consequently, (A), (H0) and (H1) hold as described in
Section~\ref{ss.nse}; recall that we require $u \in \kh_1(\N_0; \C)$
to ensure that $H(u)$ is finite.  The nonlinearity $B$ from
\eqref{e.pde} is then defined by
\begin{equation}
  \label{e.B-sharpEx}
  B(u)_k = 
  \begin{cases}
  \begin{pmatrix}
     w_0 \Re \partial f(u)/\partial w_k \\
    -w_0 \Re \partial f(u)/\partial v_k
  \end{pmatrix}_{\vphantom\int} & \text{for } k \geq 1 \,, \\
  \begin{pmatrix}
     \Re f(u) \\ 0
  \end{pmatrix} & \text{for } k=0 \,.
  \end{cases}
\end{equation}
We wish to satisfy (B0--2) and (H2--4) with $\kD_k = \kB_R^{\kY_k}(0)$
and $\kD_{\tau,L} = \kB_R^{\kY_{\tau,L}}(0)$ for some positive $R$,
$\tau$, and $L$.  Let us look at one such $f$ where all interactions
are between triples of consecutive wave numbers, namely
\begin{equation*}
  f(u) = \sum_{j=2}^\infty u_{j-1} \, u_j \, u_{j+1}.
\end{equation*}
Due to the H\"older inequality,
\begin{align*}
  \lvert f(u) \rvert
  \leq \sum_{j=2}^\infty \lvert u_{j-1} \, u_j \, u_{j+1} \rvert
  \leq \norm[2]{u}{\kl_2} \, \norm{u}{\kl_\infty}
  \leq \norm[3]{u}{\kl_2} \,.
\end{align*}
so that $f \colon \kl_2(\N_0;\C) \to \C$ is a bounded trilinear form
on $\kl_2$, hence analytic as a function on $\kl_2$ and therefore also
on $\kh_1$.

Similarly, we can check that $B(u)$ defined as in \eqref{e.B-sharpEx}
is an analytic map from $\kh_k$ to itself for all $k\in \N_0$ so that
(B0), (B1), and (B2) hold for any $L\geq 0$, $\tau>0$, and $q=2$.  As
in Section~\ref{ss.nse}, we consider the case $q=2$ because this
Gevrey space consists of sequences $\{u_k\}$ whose Fourier series
$u(x) = \sum u_k \e^{\i kx}$ is analytic in $x$.

Let us now consider the resonance condition \eqref{e.3Res} for our
example where $\omega_k=k^2$.  We obtain
\[
  4 \, ((k-1)^2 + k^2 + (k+1)^2) 
  = h^2 \, (k-1)^2 \, k^2 \, (k+1)^2
\]
or, after simplification,
\[
  12 + \frac8{k^2} = h^2 \, (k^2-1)^2 \,.
\]
For $h$ sufficiently small, there is exactly one positive root $k(h)$.
Clearly, $k(h) = O(h^{-1/2})$ as $h \to 0$, and there is a resonance
whenever $k(h) \in \Z$.

At a resonance, the embedding error in the $v_0$ component is given by
the left hand side of \eqref{e.coeff}.  For our example,
$c_\alpha=1$, so that the embedding error reads
\[
  e(h) = h \Re \bigl( 
                 \mg_{k-1} \, \mg_{k} \, \mg_{k+1} \, 
                 u^0_{k-1} \, u^0_{k} \, u^0_{k+1}
               \bigr) \,.
\]
Since $k=O(h^{-1/2})$, \eqref{e.gk} shows that $\mg_{j} = O(1)$ for
$j=k-1,k,k+1$.  Clearly, the embedding error can only be exponentially
small provided $u^0_k$ decays exponentially as $k \to \infty$.  As a
specific example, take
\[
  u^0_k = \frac{\e^{-\tau k}}{k^{\ell+2}}
\]
for $k \geq 1$ and $u_0^0=0$.  It is easily seen that $u^0 \in
\kY_{\tau, \ell}$ for any $\tau,\ell\geq 0$.  For this initial
condition, the embedding error satisfies
\begin{equation}
  \label{e.ex-embedError}
  e(h) 
  =  O \biggl(
         \frac{h \, \e^{-3\tau k}}{k^{3\ell+6}}
       \biggr)
  = O \bigl( \e^{- c h^{-1/2}} \bigr)
\end{equation}
for some $c>0$ as in \cite{MaSch2001}.

\begin{remark}
\label{r.nonlocal} 
If, in the setting above, we choose
\[
  f(u) = 1/u_0 + \sum_{j=2}^\infty u_{j-1} \, u_j \, u_{j+1} \,,
\]
then the nonlinearity $B$ defined by $f(u)$ via \eqref{e.B-sharpEx}
can be considered on the open half-balls
\[
  \kD_k = \interior(\kB_R^{\kY_k}(0)) 
        \cap \{ u \colon u_0> \epsilon\} 
\]
for some $\epsilon>0$ fixed.  Then $\kD_k$ is convex, hence
star-shaped with respect to any $u \in \kD_k$.  Defining
$\kD_{\tau,\ell}$ analogously, we obtain a nested domain hierarchy on
which $B$ satisfies (B0--2) and (H2--4); (H0), (H1), and (A) hold true
as before.  This example shows that it makes sense to consider domains
$\kD_k$ different from balls.
\end{remark}



\section*{Acknowledgments}
 
The work of C.W. was supported by the Nuffield Foundation, the
Leverhulme Foundation, and by EPSRC grant EP/D063906/1.  M.O. was
supported by a Max--Kade Fellowship, by the ESF network Harmonic and
Complex Analysis and Applications (HCAA), and by German Science
Foundation grant OL 155/5-1.

 
\bibliographystyle{plain}

\end{document}